\documentclass[11pt]{article}
\usepackage[T1]{fontenc}
\usepackage[english]{babel}
\usepackage{setspace}

\newcommand{\Keywords}[1]{\par
{\small{\em Keywords\/}: #1}}
\usepackage{abstract}

\usepackage{amssymb}
\usepackage{amsmath}
\usepackage{amsfonts}
\usepackage{amsthm}
\usepackage{graphicx}
\usepackage{subfig}

\usepackage[titletoc]{appendix}

\usepackage[]{natbib}
\bibpunct{(}{)}{;}{}{,}{,}

\usepackage{dsfont}
\newtheorem{thm}{Theorem}[section]
\newtheorem{cor}[thm]{Corollary}

\newtheorem{prop}[thm]{Proposition}
\newtheorem{lem}[thm]{Lemma}
\newcommand{\ben}{\vspace{0mm}\begin{equation}}
\newcommand{\een}{\vspace{0mm}\end{equation}}
\newcommand{\be}{\vspace{0mm}\begin{equation*}}
\newcommand{\ee}{\vspace{0mm}\end{equation*}}
\newcommand{\ba}{\vspace{0mm}\begin{equation*}\begin{aligned}}
\newcommand{\ea}{\vspace{0mm}\end{aligned}\end{equation*}}
\newcommand{\ban}{\vspace{0mm}\begin{equation}\begin{aligned}}
\newcommand{\ean}{\vspace{0mm}\end{aligned}\end{equation}}

\usepackage{authblk}

\title{\begin{center}\begin{LARGE}
Stochastic modeling of density-dependent diploid populations and extinction vortex
\end{LARGE}
\end{center}}
\author[1]{Camille Coron \thanks{coron@cmap.polytechnique.fr, corresponding author}}
\affil[1]{CMAP, École Polytechnique, CNRS UMR 7641, Route de Saclay, 91128
 Palaiseau Cedex, France}
\date{}
\begin{document}
%  \fontfamily{ptm} \selectfont
\maketitle

\begin{abstract}
We model and study the genetic evolution and conservation of a
population of diploid hermaphroditic organisms, evolving
continuously in time and subject to resource competition. In the
absence of mutations, the population follows a $3$-type nonlinear
birth-and-death process, in which birth rates are designed to
integrate Mendelian reproduction. We are interested in the long term
genetic behaviour of the population (adaptive dynamics), and in
particular we compute the fixation probability of a slightly non-neutral allele in the absence of mutations, which involves finding the unique sub-polynomial solution of a nonlinear $3$-dimensional
recurrence relationship. This equation is simplified to a $1$-order
relationship which is proved to admit exactly one bounded solution.
Adding rare mutations and rescaling time, we study the successive mutation fixations
in the population, which are given by the jumps of a limiting Markov process on the genotypes space.
At this time scale, we prove that the fixation rate of deleterious
mutations increases with the number of already fixed mutations,
which creates a vicious circle called the extinction vortex.
\end{abstract}


\section{Introduction}\label{sectioncadre}
Our goal is to model a finite population with diploid reproduction
and competition. We specially want to understand the role of
diploidy and Mendelian reproduction on mutation fixation probabilities and
on the genetic evolution of a population. We are interested
in studying the progressive accumulation of small
deleterious mutations which generates an extinction vortex in small populations (see
\citet{GilpinSoule1986, Lynch1990} and \citet{Coronbio2012} for more biological context and analyses).

The population follows a birth-and-death
process in which each individual has a natural death rate that
depends on its genotype (Section \ref{sectionecologique}). Birth rates are designed to model the
Mendelian reproduction, and individuals are competing against each
other. First, in the absence of mutation, we focus on one gene and
compute the fixation probability of an allele $a$ competing against a resident allele $A$ (Sections \ref{sectionfixation} and \ref{sectionpreuve}) as done in
\citet{ChampagnatLambert2007} for the simpler haploid case. We first consider the neutral case, where individuals all have same birth, natural death and competition death rates (i.e. alleles $A$ and $a$ are exchangeable). Here a martingale argument proves that the fixation probability of allele $a$ is simply equal to the initial proportion of this allele in the population. We next consider the case where allele $a$ is slightly non-neutral, i.e. natural death rates slightly deviate from the neutral case. Here we prove that the fixation probability of allele $a$ is differentiable in the parameters of deviation from the neutral case and that its partial derivatives are the unique subpolynomial solutions of Dirichlet problems. These equations consist in $3$-dimensional nonlinear double recurrence relationships which we manage to simplify to a $1$-dimensional double
recurrence admitting a unique bounded solution. In Section \ref{sectionmutationnelle}, we add rare mutations and rescale time in order to observe mutation apparitions. At this time scale, mutations get fixed or disappear instantaneously, and the successive fixations of mutations are given by the jumps of a Markov process $S$ on the genotypes space, called the ``Trait Substitution Sequence'', introduced by \citet{Metzetal1996} and studied notably in \citet{Champagnat2006} and \citet{ ColletMeleardMetz2011} in the diploid case. Here the population size remains finite, and we do not use any deterministic approximation. We finally get interested in the successive jump rates of $S$ in the particular case of deleterious mutations (Section \ref{sectionvortex}). Indeed we prove that when every mutation is deleterious, the Markov process $S$ jumps more and more rapidly, i.e. the fixation rate of a deleterious mutation increases  with the number of already fixed mutations, if the population is small enough which creates a vicious circle called the extinction vortex (see \citet{Coronbio2012} for biological interpretations and numerical results).

\section{Presentation of the model}\label{sectionecologique}

We consider a population of diploid hermaphroditic self-incompatible
organisms, characterized by their genotypes. Building on works of
\citet{ChampagnatFerriereMeleard2006, Champagnat2006} and \citet{ ColletMeleardMetz2011}, we
consider a birth and death process with mutation,
selection and competition under different time scales and we add diploidy. Each individual is characterized by its genotype $x\in\mathbf{G}:=\{\{\mathcal{A},\mathcal{C},\mathcal{G},\mathcal{T}\}^{G}\}^2$
where $G$ is the genome size and $\mathcal{A}$, $\mathcal{C}$,
$\mathcal{G}$, and $\mathcal{T}$ are the four nucleotides that
compose DNA. Genotype $x=(x_1,x_2)$ is in fact composed with two DNA
strands $x_1$ and $x_2$ in $\{\mathcal{A},\mathcal{C},\mathcal{G},\mathcal{T}\}^{G}$. In
Sections \ref{sectionecologique} to \ref{sectionpreuve}, we consider the case without mutation and assume that the population is initially composed with
individuals that only differ from each other on one gene. For this
gene, there are two possible alleles, denoted by $A$ and $a$ in $\{\mathcal{A},\mathcal{C},\mathcal{G},\mathcal{T}\}^{G'}$ where $G'\leqslant G$.
The genotypes of individuals are thus denoted $AA$, $Aa$, and $aa$, and we represent the population dynamics by
the Markov process: \be Z:t\mapsto Z_t=(k_t,m_t,n_t),\ee

that gives the respective numbers of individuals with genotype $AA$,
$Aa$, and $aa$ at time $t$. For more simplicity, we will also refer
to these genotypes as types $1$, $2$, and $3$. We assume that the
process $Z$ is a birth-and-death process with competition on
$\mathbb{N}^3$, and we now detail the birth and death rates of
individuals of each genotype. The population has maximum fecundity
rate $r$. More precisely, if the population contains $N$
individuals, $rN$ is the rate at which two distinct individuals of
the population encounter, and the maximum total birth rate. These
two individuals are chosen uniformly randomly in the population, and
their encounter gives rise to a birth with a probability $p_{ij}$
($p_{ij}=p_{ji}$) that depends on their two genotypes $i$ and $j$.
$p_{ij}$ can be defined biologically as the selective value
associated with the couple of genotypes $i$ and $j$, and represents
both the degree of adaptation of types $i$ and $j$ and their
compatibility. Finally the new-born individual results from a
segregation (genetic melting between the genotypes of its parents),
satisfying Mendel's laws of heredity. Then in the population
$Z=(k,m,n)$ such that $k+m+n\geqslant2$, if we define
$b_{ij}:=rp_{ij}$, the rate $b_i(Z)$ at which an individual of type
$i\in\{1,2,3\}$ arises is: \ban \label{birthrates}
b_1(Z)&=b_{11}\frac{k(k-1)}{N-1}+b_{12}\frac{km}{N-1}+b_{22}\frac{m(m-1)}{4(N-1)},\\
b_2(Z)&=b_{12}\frac{km}{N-1}+b_{22}\frac{m(m-1)}{2(N-1)}+b_{23}\frac{mn}{N-1}+
b_{13}\frac{2kn}{N-1},\\
b_3(Z)&=b_{33}\frac{n(n-1)}{N-1}+b_{23}\frac{mn}{N-1}+b_{22}\frac{m(m-1)}{4(N-1)}.\ean

Note that if the population $Z$ has size $N$,
\ben\label{conditionb} b_1(Z)+b_2(Z)+b_3(Z)\leqslant rN.\een

We assume self-incompatibility, which implies that when the population size reaches $1$, no
birth can occur anymore and the population can be considered as
extinct. Now individuals can
die either naturally or due to competition with others. We
denote by $d_i$ the natural death rate of individuals with type $i$
and $c_{ij}$ the competition rate of $i$ against $j$, i.e. the rate
at which a fixed individual of type $i$ makes a fixed individual of
type $j$ die. We assume
\ben\label{conditionc}c_{ij}>0 \quad\quad\forall i,j\in\{1,2,3\},\quad\text{i.e.}\quad\underline{c}=\underset{i,j\in
\{1,2,3\}}{\inf}c_{ij}>0\een

and that when the population size reaches $2$, no death can occur,
hence the population cannot get extinct. We then denote the state space of $Z$ by
$$\mathbb{N}^3_{**}=\mathbb{N}^3\setminus\{(0,0,0),(1,0,0),(0,1,0),(0,0,1)\}.$$ In the population $Z=(k,m,n)$ such that
$k+m+n\geqslant3$, the rate $d^{(i)}(Z)$ at which the population
loses any individual of type $i$ then is:
\ban \label{deathrates}d^{(1)}(Z)&=(d_1+c_{11}(k-1)+c_{21}m+c_{31}n)k,\\
d^{(2)}(Z)&=(d_2+c_{12}k+c_{22}(m-1)+c_{32}n)m,\\
d^{(3)}(Z)&=(d_3+c_{13}k+c_{23}m+c_{33}(n-1))n,\ean

and if $k+m+n=2$,
\ben \label{deathrates2}d^{(1)}(Z)=d^{(2)}(Z)=d^{(3)}(Z)=0.\een

From $\eqref{conditionb}$, $\eqref{conditionc}$, and Theorem $2.7.1$
in \citet{Norris1997}, the process $Z$ does not explode. Then $Z_{t}$
is defined for all $t>0$, and we denote by $\mathbb{P}_{(k,m,n)}$
the law of $Z$ starting from state $(k,m,n)$, $\mathbb{E}_{(k,m,n)}$
the associated expectation, $(\mathcal{Z}_l)_{l\in\mathbb{N}}$ the
embedded Markov chain, and $(\mathcal{F}_l)_{l\in\mathbb{N}}$ the
filtration generated by $\mathcal{Z}$.

\textbf{Notation:} For every other process $X$,
$\mathbb{P}^{X}_{X_0}$ is the law of $X$ starting from $X_0$, and
$\mathbb{E}^{X}_{X_0}$ is the associated expectation. If $X$ is a
continuous-time (resp. discrete time) process, we denote $T^X_x$
(resp. $\mathcal{T}_x^X$) the reaching time of $x$ by $X$.

In the following, the population size process will play a main role;
we define $N:t\mapsto N_t=(k_t+m_t+n_t)$ where $Z_t=(k_t,m_t,n_t)$,
for every time $t>0$ and $(\mathcal{N}_l)_{l\in\mathbb{N}}$ the
embedded Markov chain. $N$ is stochastically dominated by the
logistic birth-and-death process $Y$ with transition rates: 

\ben
\label{transitionY} a_{ij}=\left\{\begin{array}{l}rj
\quad\quad\text{ if $j=i+1$, }\\\underline{c}j(j-1)
\quad\quad\text{if $j=i-1$ and $i\neq2$,}\\0
\quad\quad\text{otherwise}\end{array}\right.\een 

We define
$\mathcal{Y}$ the embedded Markov chain.

\begin{prop}\label{proptemps}
For all $N\in\mathbb{N}$, there exists $\rho>0$ such that
$\mathbb{E}_N((1+\rho)^{\mathcal{T}^{\mathcal{Y}}_2})<\infty$.
\end{prop}

\begin{proof} Let $N_0$ be such that $b<(d+c(N_0-1))$.
We assume that $N>N_0$, without loss of generality. Note that it
suffices to prove that for every integer $n\in[3,N]$, there exists
$\rho_n>0$ such that
\ben\label{inductionT}\mathbb{E}_n\left((1+\rho_n)^{T^{\mathcal{Y}}_{n-1}}\right)<\infty.\een

Indeed, $\mathbb{E}_N((1+\rho)^{\mathcal{T}^{\mathcal{Y}}_2})
=\prod_{j=3}^N\mathbb{E}_j((1+\rho)^{T^{\mathcal{Y}}_{j-1}}<\infty$
if $\rho\leqslant\inf_i\rho_i$. Now, from \citet{Seneta1966} p.
$428$, \eqref{inductionT} is true for $n=N$, since $N>N_0$. Now,
following the proof of Lemma $5.11$ of \citet{ColletMartinez2012},
let us prove by induction that if \eqref{inductionT} is true for
$n+1$ then it is also true for $n$. We assume that
\eqref{inductionT} is true for $n+1$ and that $Y_0=n$, and we define
$M$ the random number of returns in $n$ before going to $n-1$. $M$
follows a geometrical law with parameter $p=b/(b+d+c(n-1))$. Then
$$\mathcal{T}^{\mathcal{Y}}_{n-1}=M+1+\sum_{i=1}^M\mathcal{T}_{n,i}$$

where the $\mathcal{T}_{n,i}$ are independent and distributed as
$\mathcal{T}^{\mathcal{Y}}_{n}$ for all $i$. Then by strong Markov
Property in the stopping times $\mathcal{T}_{n,i}$, we obtain
$$\mathbb{E}_n((1+\rho)^{\mathcal{T}^{\mathcal{Y}}_{n-1}})\leqslant\sum_{m=0}^{\infty}
\left(\mathbb{E}_{n+1}\left((1+\rho)^{\mathcal{T}^{\mathcal{Y}}_n+2}\right)\right)^m(1-p)p^m.$$

Finally, since $\eqref{inductionT}$ is true for $n+1$, from the
Dominated Convergence Theorem,
$\mathbb{E}_{n+1}\left((1+\rho)^{\mathcal{T}^{\mathcal{Y}}_n+2}\right)$
goes to $1$ when $\rho$ goes to $0$, hence there exists $\rho_{n-1}$
such that
$\mathbb{E}_{n+1}\left((1+\rho_{n-1})^{\mathcal{T}^{\mathcal{Y}}_n+2}\right)<1/p$
which gives the result. 
\end{proof}

\medskip

\begin{prop}\label{sup} For all $p\geqslant1$, if $\mathbb{E}(N_0^p)<\infty$
then $\underset{t\geqslant0}{\sup} \,\mathbb{E}(N_t^p)<\infty$. \end{prop}

\begin{proof} We set $Y_0=N_0$. It suffices to prove that
$\sup_t\mathbb{E}(Y_t^p)<\infty$. $(Y_t)_{t>0}$ is a recurrent,
irreducible, and ergodic Markov process on
$\mathbb{N}\setminus\{0,1\}$, with stationary law $l$ (see Equation
\eqref{formulel} for a more general case), and we can easily check
that $E_p:=\displaystyle{\sum_{j=2}^{\infty}}l(j)j^p<\infty$ for all
$p$. Now let us define the Markov process $(Y_t,Z_t)_{t\geqslant0}$
such that $Y$ and $Z$ have same transition rates, are independent,
and $Z_0$ has law $l$. We define
$(\mathcal{Y}_n,\mathcal{Z}_n)_{n\in\mathbb{N}}$ the associated
Markov chain, and
$\mathcal{T}=\inf\{n|\mathcal{Y}_n=\mathcal{Z}_n\}$. Following the
proof of Theorem $6.6.4$ in \citet{Durrett2010} p. $308$, we have
\ba\left|\mathbb{E}(\mathcal{Y}_n^p)-E_p\right|=\left|\mathbb{E}(\mathcal{Y}_n^p)-\mathbb{E}(\mathcal{Z}_n^p)\right|&\leqslant\sum_{z\geqslant2}z^p|\mathbb{P}(\mathcal{Y}_n=z)-\mathbb{P}(\mathcal{Z}_n=z)|\\&
\leqslant\sum_{z\geqslant2}z^p(\mathbb{P}(\mathcal{Y}_n=z,\mathcal{T}>n)+\mathbb{P}(\mathcal{Y}_n=z,\mathcal{T}>n))\\&=
\mathbb{E}((\mathcal{Y}_n^p+\mathcal{Z}_n^p)\mathbf{1}_{\mathcal{T}>n})\\&\leqslant2\mathbb{E}(\mathcal{Y}_n^p\mathbf{1}_{\mathcal{T}>n}\mathbf{1}_{Y_0>Z_0})+2\mathbb{E}(\mathcal{Z}_n^p\mathbf{1}_{\mathcal{T}>n}\mathbf{1}_{Z_0>Y_0})\\&\leqslant2\mathbb{E}(\mathcal{Y}_n^p\mathbf{1}_{\mathcal{T}_2^{\mathcal{Y}}>n})+2\mathbb{E}(\mathcal{Z}_n^p\mathbf{1}_{\mathcal{T}_2^{\mathcal{Z}}>n}).\ea

Now
\ba\mathbb{E}(\mathcal{Y}_n^p\mathbf{1}_{\mathcal{T}_2^{\mathcal{Y}}>n})&\leqslant\sum_{z\geqslant2}(z+n)^p\mathbb{P}(\mathcal{T}^{\mathcal{Y}}_2>n;Y_0=z)\\&\leqslant2^p\sum_{z\geqslant
n}z^p\mathbb{P}(Y_0=z)+2^p\sum_{2\leqslant
z<n}n^p\mathbb{P}(\mathcal{T}_2^{\mathcal{Y}}\geqslant
n;Y_0=z)\\&\leqslant2^p\sum_{z\geqslant
n}z^p\mathbb{P}(Y_0=z)+2^pn^p\mathbb{P}(\mathcal{T}_2^{\mathcal{Y}}\geqslant
n).\ea

From Proposition $\ref{proptemps}$, and since $\mathbb{E}(Y_0^p)<\infty$, 
$n^p\mathbb{P}(\mathcal{T}_2^Y\geqslant n)$ and $\sum_{z\geqslant
n}z^p\mathbb{P}(Y_0=z)$ converge to $0$. Then $\mathbb{E}(Y_n^p)$
converges to $E_p$ when $n$ goes to infinity. Since $Y$ does not
explode and $\mathbb{E}(Y_0^p)<\infty$, we have
$\underset{t}{\sup}\,\mathbb{E}(Y_t^p)<\infty$.
\end{proof}

\section{Fixation probabilities}\label{sectionfixation}

\subsection{Absorbing states}
The birth and death process $Z$ admits the following absorbing states sets:
\begin{itemize}
\item $\Gamma_a=\{(0,0,n), n\geqslant2\}$ is the set of states for which allele $a$
is fixed and allele $A$ has disappeared.
\item $\Gamma_A=\{(k,0,0), k\geqslant2\}$ is the set of states for which allele $A$
is fixed and allele $a$ has disappeared.
\item $\Gamma:=\Gamma_a\cup\Gamma_A$
\end{itemize}

We are interested in computing the probability that allele $a$ goes
to fixation (i.e. $Z$ reaches $\Gamma_a$), when $Z$ starts from any
state $(k,m,n)$. We now define $\mathcal{T}_{\Omega}$ the (discrete) reaching time of set $\Omega$
by $\mathcal{Z}$ for all $\Omega\subset\mathbb{N}^3_{**}$.
The following result is an adaptation of Proposition $6.1.$ in
\citet{ChampagnatLambert2007} to the diploid case.

\begin{prop}\label{proptempsdarret} There exists a constant $C$ such
that for any initial state $(k,m,n)$ in $\mathbb{N}^3_{**}$,
$\mathbb{E}_{(k,m,n)}(\mathcal{T}_{\Gamma})\leqslant C(k+m+n).$
\end{prop}

\begin{proof} Let $\mathcal{T}_{\{2\}}$ be the first time where
the Markov chain $\mathcal{N}$ reaches $2$ (or returns to $2$ if
$\mathcal{N}_0=2$), and define
$$\overline{\mathcal{T}}_{\{2\}\rightarrow\Gamma}:=
\sup_{\substack{(k,m,n)|\\k+m+n=2}}\mathbb{E}_{(k,m,n)}(\mathcal{T}_{\Gamma}).$$
Then $\mathbb{E}_{(k,m,n)}(\mathcal{T}_{\Gamma})\leqslant
\mathbb{E}_{(k,m,n)}(\mathcal{T}_{\{2\}})
+\overline{\mathcal{T}}_{\{2\}\rightarrow\Gamma},$ and
$\overline{\mathcal{T}}_{\{2\}\rightarrow\Gamma}$ is independent of
$(k,m,n)$. We prove first that
$\overline{\mathcal{T}}_{\{2\}\rightarrow\Gamma}<\infty$ and second
that there exists a constant $C_1$ such that
$\mathbb{E}_{(k,m,n)}(\mathcal{T}_{\{2\}})<C_1(k+m+n)$ for all
$(k,m,n)$ in $\mathbb{N}^3_{**}$. Now,

\ban\label{overlineT2}\overline{\mathcal{T}}_{\{2\}\rightarrow\Gamma}&=
\sup_{\substack{(k,m,n)|\\k+m+n=2}}\mathbb{E}_{(k,m,n)}\left(\mathcal{T}_{\Gamma}
\mathbf{1}_{\{\mathcal{T}_{\{2\}}\geqslant
\mathcal{T}_{\Gamma}\}}+\mathcal{T}_{\Gamma}\mathbf{1}_{\{\mathcal{T}_{\{2\}}
<\mathcal{T}_{\Gamma}\}}\right)\\
&\leqslant\sup_{\substack{(k,m,n)|\\k+m+n=2}}\mathbb{E}_{(k,m,n)}(\mathcal{T}_{\{2\}})
+\sup_{\substack{(k,m,n)|\\k+m+n=2}}\mathbb{E}_{(k,m,n)}\left((\mathcal{T}_{\Gamma}-\mathcal{T}_{\{2\}})
\mathbf{1}_{\{\mathcal{T}_{\{2\}}<\mathcal{T}_{\Gamma}\}}\right)\\&
\leqslant\sup_{\substack{(k,m,n)|\\k+m+n=2}}\mathbb{E}_{(k,m,n)}(\mathcal{T}_{\{2\}})\\&
+\sup_{\substack{(k,m,n)|\\k+m+n=2}}\sum_{\substack{(k',m',n')|\\k'+m'+n'=2}}\mathbb{E}_{(k,m,n)}
\left((\mathcal{T}_{\Gamma}-\mathcal{T}_{\{2\}})
\mathbf{1}_{\{\mathcal{T}_{\{2\}}<\mathcal{T}_{\Gamma}\}}\mathbf{1}_{Z_{T_{\{2\}}}=(k',m',n')}\right)
\\&
\leqslant \sup_{\substack{(k,m,n)|\\k+m+n=2}}\mathbb{E}_{(k,m,n)}(\mathcal{T}_{\{2\}})
+\overline{\mathcal{T}}_{\{2\}\rightarrow\Gamma}\sup_{\substack{(k,m,n)|\\k+m+n=2}}\mathbb{P}_{(k,m,n)}
(\{\mathcal{T}_{\{2\}}
< \mathcal{T}_{\Gamma}\}),\ean

where the last inequality is obtained using the strong Markov property in $\mathcal{T}_{\{2\}}$. Defining
\ba \overline{p}&=\sup_{(k,m,n)|k+m+n=2}\mathbb{P}_{(k,m,n)}(\mathcal{T}_{\{2\}}< \mathcal{T}_{\Gamma}) \quad\text{and} \\
\overline{\mathcal{T}}_{\{2\}\rightarrow\{2\}}&=\sup_{(k,m,n)|k+m+n=2}\mathbb{E}_{(k,m,n)}(\mathcal{T}_{\{2\}}),\ea

we have $\overline{p}<1$, since for every $(k,m,n)$ such that
$k+m+n=2$, there exists a path for $Z$ starting from $(k,m,n)$ and
reaching $\Gamma$ before reaching the set $\{N=2\}$. Besides,
$\overline{\mathcal{T}}_{\{2\}\rightarrow\{2\}}$ is bounded by the
expectation of the mean time of coming back in $\{N=2\}$ for the
process $Y$ defined by Equation $\eqref{transitionY}$. So
$\overline{\mathcal{T}}_{\{2\}\rightarrow\{2\}}<\infty$, from
Theorem $3.3.3$ of \citet{Norris1997}. Finally, from
$\eqref{overlineT2}$,
$(1-\overline{p})\overline{\mathcal{T}}_{\{2\}\rightarrow\Gamma}\leqslant
\overline{\mathcal{T}}_{\{2\}\rightarrow\{2\}},$ then
$\overline{\mathcal{T}}_{\{2\}\rightarrow\Gamma}<\infty$.
Now, let us consider the Markov chain $(\mathcal{Y}_n)_{n\in\mathbb{N}}$ on
$\mathbb{N}\setminus\{0,1\}$, associated with $Y$. $\mathcal{N}$ being stochastically dominated by $\mathcal{Y}$, if $N=k+m+n$,
$\mathbb{E}_{(k,m,n)}(\mathcal{T}_{\{2\}})
\leqslant \mathbb{E}^\mathcal{Y}_N(\inf\{n|\mathcal{Y}_n=2\}).$ Define $S_{N,i}=\mathbb{E}^\mathcal{Y}_N(\inf\{n|\mathcal{Y}_n=i\})$ and let $N_0\geqslant2$ be a natural integer such that
$\frac{\overline{b}}{\overline{b}+\underline{c}N_0}\leqslant\frac{1}{3}$.
If $N\geqslant N_0$ then $S_{N,2}=S_{N,N_0}+S_{N_0,2}.$ Moreover, since
$\frac{\overline{b}}{\overline{b}+\underline{c}N}\leqslant\frac{1}{3}$
for all $N\geqslant N_0$, $S_{N,N_0}\leqslant\mathbb{E}(U_{N,N_0})$
where $U_{N,i}$ is the first reaching time of $i$, for the discrete
time random walk on $\mathbb{Z}$ starting from $N$ and having
probability $1/3$ to jump one step on the right and $2/3$ to jump one
step on the left, for every state. We know that
$\mathbb{E}(U_{N,N_0})=3(N-N_0)$ \citet{Norris1997}, pp. $21-22$. So if $N\geqslant N_0$,
$\mathbb{E}(S_{N,2})\leqslant\mathbb{E}(S_{N_0,2})+3(N-N_0)$. Then
there exists a constant $C_1>0$ such that $\mathbb{E}(S_{N,2})<C_1N$ for
all $N\geqslant2$. \end{proof}

\medskip
We now consider the fixation probabilities of allele $a$ as a function of the initial state of the population. We define $F_a=\{(Z_t)_{t>0} \text{ reaches }
\Gamma_a\}$ and $u(Z)=\mathbb{E}_Z(\mathbf{1}_{F_a})$ is the fixation
probability of allele $a$ knowing that the population starts from
state $Z$. $u$ also depends on the demographic parameters of the
population, and this dependence will be explicitely written down when
necessary. Note that $(u(Z_t))_{t>0}$ is a martingale since
\ben\label{umartingale} u(Z_t)=u(k_t,m_t,n_t)=\mathbb{E}_{Z_t}(\mathbf{1}_{F_a})=\mathbb{E}(\mathbf{1}_{F_a}|\mathcal{F}_t).\een

In the neutral case (Section \ref{sectioncasneutre}), a martingale
argument gives us the value of $u$, and in the non-neutral case
with small mutation assumption (Section \ref{sectioncasnonneutre}),
we prove that $u$ admits a Taylor expansion in the parameters of
deviation from the neutral case.

\subsection{Neutral case}\label{sectioncasneutre}
We now consider the neutral case when ecological parameters do not
depend on genotypes, i.e. when $b_{ij}=b$, $c_{ij}=c$, and $d_i=d$
for all $i$ and $j$ in $\{1,2,3\}$. We first prove the

\begin{prop} \label{casneutre}
In the neutral case, for all $(k,m,n)$ in $\mathbb{N}^3_{**}$ and for all ecological parameters $b$, $d$ and $c$,
$$u(k,m,n)=\frac{m+2n}{2(k+m+n)}.$$
\end{prop}

\begin{proof} Let us define the function $p:(k,m,n)\mapsto
(m+2n)/2(k+m+n)$ and denote by $T_l$ the $l$-th jump time of the
population (i.e. the time at which occurs the $l$-th event, birth or
death).  The Markov chain $(p(\mathcal{Z}_l))_{l\in\mathbb{N}}$
gives the successive proportions of allele $a$ in the population. We
now prove that $p(\mathcal{Z}_l))_{l\in\mathbb{N}}$ is a
$\mathcal{F}_{l}$-bounded martingale. To this aim, we distinguish
two types of states: those where the population size is greater or
equal to $3$ and those where it is equal to $2$. For
$\mathcal{Z}_l=(k_l,m_l,n_l)$ such that $\mathcal{N}_l\geqslant3$,
one can compute $\mathbb{E}(p(\mathcal{Z}_{l+1})|\mathcal{Z}_l)$ by
decomposing it according to the nature of the $l+1$-th event:
\ba
\mathbb{E}(p(\mathcal{Z}_{l+1})|\mathcal{Z}_l)&=\frac{2\mathcal{N}_lp(\mathcal{Z}_l)-2}{2\mathcal{N}_l-2}\mathbb{P}(
death\text{ } of \text{  }aa)
+\frac{2\mathcal{N}_lp(\mathcal{Z}_l)-1}{2\mathcal{N}_l-2}\mathbb{P}(death\text{
} of\text{  }
Aa)\\&+\frac{2\mathcal{N}_lp(\mathcal{Z}_l)}{2\mathcal{N}_l-2}\mathbb{P}(death
\text{ }of\text{  }
AA)+\frac{2\mathcal{N}_lp(\mathcal{Z}_l)+2}{2\mathcal{N}_l+2}\mathbb{P}(birth
\text{ }of\text{  }
aa)\\&+\frac{2\mathcal{N}_lp(\mathcal{Z}_l)+1}{2\mathcal{N}_l+2}\mathbb{P}(birth\text{
} of\text{ }
Aa)+\frac{2\mathcal{N}_lp(\mathcal{Z}_l)}{2\mathcal{N}_l+2}\mathbb{P}(birth
\text{ }of\text{  } AA)\\&=p(\mathcal{Z}_l).\ea

The same result can be easily proved for $N_l=2$.
\medskip
From Doob's stopping time theorem applied to the bounded martingale
$(p(\mathcal{Z}_l))_{l}$ and to the stopping time
$\mathcal{T}_{\Gamma}$ (a.s. finite, from Proposition
\ref{proptempsdarret}), we get:
\be\mathbb{E}_{k,m,n}(p(\mathcal{Z}_{\mathcal{T}_{\Gamma}}))=\frac{2n+m}{2(k+m+n)}.\ee
Now
\ba\mathbb{E}_{k,m,n}(p(\mathcal{Z}_{\mathcal{T}_{\Gamma}}))&=
\mathbb{E}_{k,m,n}(p(\mathcal{Z}_{\mathcal{T}_{\Gamma}})\mathbf{1}_{T_{\Gamma_a}<T_{\Gamma_A}})
+\mathbb{E}_{k,m,n}(p(\mathcal{Z}_{\mathcal{T}_{\Gamma}})\mathbf{1}_{T_{\Gamma_a}>T_{\Gamma_A}})
\\&=\mathbb{P}_{k,m,n}(T_{\Gamma_a}<T_{\Gamma_A})=u(k,m,n)\ea
since
$\mathbb{E}_{k,m,n}(p(\mathcal{Z}_{\mathcal{T}_{\Gamma}})|T_{\Gamma_a}<T_{\Gamma_A})=1$
and
$\mathbb{E}_{k,m,n}(p(\mathcal{Z}_{\mathcal{T}_{\Gamma}})|T_{\Gamma_a}>T_{\Gamma_A})=0$.\end{proof}

\medskip
When the mutation is
not neutral, we do not obtain any closed formula for $p(Z)$ as
previously. We instead consider the Dirichlet problem satisfied by $u$.

\subsection{Deviation from the neutral case}\label{sectioncasnonneutre}

\subsubsection{A Dirichlet Problem}

We now arbitrarily assume that allele $a$ is slightly deleterious,
i.e. the demographic parameters $(b_{ij})_{i,j}$, $(c_{ij})_{i,j}$,
and $(d_i)_{i}$ are less advantageous for genotypes $Aa$ and
$aa$ than for genotypes $AA$, and slightly deviate from the
neutral case. This latter assumption (small mutation sizes) is
justified in biology papers such as \citet{Orr1998,Orr1999} which show that species evolution is partly due to the
fixation of a large number of small mutations. Besides, we assume that carrying allele $a$ only
influences the natural death rate of individuals. More precisely, we
set
\ban \label{equationtaux} b_{ij}&=b \quad\forall i,j,\\ c_{ij}&=c, \quad\forall i,j, \text{ whereas }\\
d_1=d, \quad d_2=&d+\delta \quad\text{ and }\quad d_3=d+\delta',\ean

where $\delta$ and
$\delta'$ are close to $0$. Note that if $\delta'$ is positive and
$\delta$ is equal to $0$, then allele $a$ is deleterious. The effect
of $\delta$ is more intricate because it affects heterozygous
individuals, with the same apparent effect on both alleles. It
simply represents a more or less important adaptation of heterozygotes
compared to $AA$ homozygotes and as we will see later (Subsection
\ref{sectiondelta}), its role in the deleterious or positive effect of allele $a$
depends on the initial genetic repartition of the population. We
denote by $L^{\delta,\delta'}$ the infinitesimal generator of $Z$
with assumptions $\eqref{equationtaux}$, and by $u((k,m,n), \delta,
\delta')$ the fixation probability of allele $a$, knowing that $Z$
starts from $(k,m,n)$, for all $(k,m,n)$ in
$\mathbb{N}^3_{**}$. We then have for all real bounded function $f$ on $\mathbb{N}^3_{**}$:

\ba (L^{\delta,\delta'}
f)(k,m,n)\!&=b_1(Z)f(k+1,m,n)+b_2(Z)f(k,m+1,n)+b_3(Z)f(k,m,n+1)
\\&+d^{(1)}(Z)f(k-1,m,n)+d^{(2)}(Z)f(k,m-1,n)+d^{(3)}(Z)f(k,m,n-1)
\\&-(bN+(d+c(N-1))N+\delta m+\delta'n)f(k,m,n).\ea

We define from $\eqref{birthrates}$, $\eqref{deathrates}$, and
$\eqref{deathrates2}$, the infinitesimal generator

\ban\label{formuleL}
(Lv)(k,m,n)&=(L^{0,0}v)(k,m,n)\\&=\frac{b}{N-1}\left[\left(k(k-1)+km+\frac{m(m-1)}{4}\right)\right.v(k+1,m,n)
\\&\phantom{\frac{b}{N-1}[}+\left(km+\frac{m(m-1)}{2}+mn+2kn\right)v(k,m+1,n)\\&\phantom{\frac{b}{N-1}[}
+\left.\left(n(n-1)+mn+\frac{m(m-1)}{4}\right)v(k,m,n+1)\right]
\\&+(d+c(N-1))\left[kv(k-1,m,n)+mv(k,m-1,n)+nv(k,m,n-1)\right]\\&-(bN+dN+cN(N-1))v(k,m,n) \quad\quad\text{ if } k+m+n\geqslant3,\\(Lv)(k,m,n)&=\frac{b}{N-1}\left[\left(k(k-1)+km+\frac{m(m-1)}{4}\right)\right.v(k+1,m,n)
\\&\phantom{\frac{b}{N-1}[}+\left(km+\frac{m(m-1)}{2}+mn+2kn\right)v(k,m+1,n)\\&\phantom{\frac{b}{N-1}[}
+\left.\left(n(n-1)+mn+\frac{m(m-1)}{4}\right)v(k,m,n+1)\right]\\&-bNv(k,m,n) \quad\quad\text{ if } k+m+n=2.\ean

Using that $(u(Z_t,\delta,\delta'))_{t\leqslant0}$ is a bounded martingale if $Z$
has infinitesimal generator $L^{\delta,\delta'}$ (Equation $\eqref{umartingale}$), we obtain the

\begin{prop}\label{equationu}
$u(.,\delta,\delta')$ satisfies:
\ben\label{Deltau}\left\{\begin{array}{l} (L^{\delta,\delta'}
u(.,\delta,\delta'))(k,m,n)=0 \quad\forall (k,m,n) | N=k+m+n\geqslant2\\
   u((0,0,n),\delta,\delta')=1 \quad\forall n\geqslant2\\
   u((k,0,0),\delta,\delta')=0 \quad\forall k\geqslant2
\end{array}\right.\een
\end{prop}

Our main result in this section is the following theorem studying in detail the deviation of $u$ from the neutral case.

\begin{thm}\label{maintheorem}
For all $(k,m,n)$ in $\mathbb{N}^3_{**}$, the function
$(\delta,\delta')\mapsto u((k,m,n),\delta,\delta')$
is an analytic function of $(\delta,\delta')$ in the neighborhood of $(0,0)$. Moreover,
\be u((k,m,n),\delta,\delta')=p(k,m,n)-\delta v(k,m,n)-\delta'v'(k,m,n)+o(|\delta|+|\delta'|),\ee where
\begin{eqnarray}\label{formulev}v(k,m,n)&=&(k-n)\left[\frac{m}{N}x_N+\frac{N^2-(k-n)^2}{N^2}y_N\right],
\\\label{formulev'}v'(k,m,n)&=&\frac{nY}{N}x_N+mx'_N+Y(2N-Y)\left(\frac{y'_N}{N}-\frac{Y}{2N^2}y_N\right).
\end{eqnarray}
The sequences $x_{N}$, $y_{N}$, $x'_{N}$, and $y'_{N}$ are defined as the unique bounded solutions of $2$-order recurrence equations (Propositions \ref{propositionv} and \ref{propositionv'}).
\end{thm}

The proof of this theorem is decomposed in several parts: the existence and formula
of the two partial derivatives is obtained in Sections
\ref{sectiondelta} and \ref{sectiondelta'} and the analyticity of $u$ is in Section \ref{sectiondifferentiabilite}. In the following subsections, we consider separately the cases where
$\delta=0$ and $\delta'=0$.

\subsubsection{The dependence of $u$ in $\delta$}\label{sectiondelta}

To simplify notations, we define:
$u((k,m,n),\delta)=u((k,m,n),\delta,0)$.
We will show that the derivative of $u$ at $\delta=0$
is the unique sub-polynomial (i.e. lower than a
polynomial function in $N=k+m+n$) solution of a nonlinear recurrence
equation in $(k,m,n)$. Such result has been obtained in \citet{ChampagnatLambert2007} for the haploid case. Here, the nonlinearity due to both competition and
diploid segregation terms generates new mathematical difficulties. We
will use some arguments developped in \citet{ChampagnatLambert2007} and will
here focus on the difficulties brought by diploidy. We say that a function $f$ on $\mathbb{N}^3$ is sublinear if there exists a constant $C$ such that $|f(k,m,n)|\leqslant C(k+m+n)$ for every $(k,m,n)$.
\begin{prop} \label{vSL} For all $(k,m,n)$ in $\mathbb{N}^3_{**}$,

$u((k,m,n),.)$ is differentiable at $0$. Its derivative $v(k,m,n)$ is the unique sublinear solution of the system of equations
\ben\left\{\begin{array}{l}(Lv)(k,m,n)=\frac{m(n-k)}{2N(N-1)}\quad\forall(k,m,n)\in\mathbb{N}^3_{**}\\v(2,0,0)=v(0,0,2)=0\end{array}\right.\label{Deltavmin}\een
\end{prop}

\begin{proof} As in the simplest case of haploid populations, we introduce
paths of $Z$, i.e. the sequence of states visited by this process.
Indeed the fixation probability of the mutant allele $a$ if the
population $Z$ starts from state $(k,m,n)$ can be written as the sum
of the probabilities of every path starting from $(k,m,n)$ and
reaching a state $(0,0,n')$ with $n'\geqslant2$. We then denote by
$S_{(k,m,n)\rightarrow\Omega}$ the set of paths linking
$(k,m,n)\notin\Gamma$ to $\Omega$ without reaching $\Gamma$ before
$\Omega$, and $(i_1,i_2,...,i_l)$ a path, $i_j$ being the $j$-th
state of the path. We finally denote by $\pi^{\delta}_{i_ji_{j+1}}$
the transition probability from state $i_j$ to state $i_{j+1}$ for
$Z$. Then
\be u((k,m,n),\delta)=\sum_{(i_1,..i_l)\in
S_{(k,m,n)\rightarrow\Gamma_a}}\pi^{\delta}_{i_1i_2}...\pi^{\delta}_{i_{l-1}i_l}.\ee
Now $\pi^{\delta}_{i_ji_{j+1}}$ is a differentiable function of $\delta$ and the absolute value of
its derivative at $\delta=0$ is bounded independently of $(k,m,n)$
by a constant denoted by $C_1$. To prove this latter assertion, we
consider separately the different possible transitions for the
population in state $(k,m,n)$. For instance the transition
probability from state $(k,m,n)$ to state $(k+1,m,n)$ is
\be\pi^{\delta}_{(k,m,n),(k+1,m,n)}=\frac{b(k(k-1)+km+m(m-1)/4)}{(N-1)(bN+dN+\delta
m+cN(N-1))}.\ee

Then $\pi^{\delta}_{(k,m,n),(k+1,m,n)}$ is
differentiable with respect to $\delta$ at $0$, and:
\ba\left|\frac{\partial\pi^{\delta}_{(k,m,n)(k+1,m,n)}}{\partial\delta}\right|_{\delta=0}&
=\frac{mb(k(k-1)+km+m(m-1)/4)}{(N-1)(bN+dN+cN(N-1))^2}
\\&\leqslant\frac{m}{bN+dN+cN(N-1)}\leqslant\frac{2}{b+d+2c}.\ea

Similar computations are made for other possible transitions. Then
$u^{\delta}_{(k,m,n)}$ is differentiable with respect to $\delta$ at
$\delta=0$ and
\ba \left|\frac{\partial
u((k,m,n),\delta)}{\partial\delta}\right|_{\delta=0}&=\sum_{\substack{(i_1,..i_l)\in\\
S_{(k,m,n)\rightarrow\Gamma_a}}}\sum_{l'=1}^{l-1}\pi^{0}_{i_1i_2}...
\pi^{0}_{i_{l'-1}i_l'}\left|\frac{\partial\pi^{\delta}_{i_{l'}i_{l'+1}}}{\partial\delta}\right|_{\delta=0}
\pi^{0}_{i_{1'+1}i_{l'+2}}...\pi^{0}_{i_{l-1}i_l}
\\&\leqslant C_1\sum_{l'\geqslant 1}\;\sum_{(k',m',n')\in\mathbb{N}^3}\;\sum_{\substack{(i_1,...,i_{l'})\in \\S_{(k,m,n)
\rightarrow(k',m',n')}}}\pi^0_{i_1i_2}...
\pi^0_{i_{l'-1}i_l'}\\&\phantom{\leqslant C_1 \sum_{l'\geqslant
1}}\times\sum_{\epsilon\in\mathbb{N}^3,\|\epsilon\|=1}\;
\sum_{\substack{l''\geqslant 0,(j_1,...,j_{l''})\in\\
S_{(k',m',n')+\epsilon\rightarrow\Gamma_a}}}\pi^0_{j_1j_2}...\pi^0_{j_{l''-1}j_{l''}}.\ea

Then,
\ba \left|v(k,m,n)\right|&\leqslant
C_1\sum_{l'\geqslant
1}\;\sum_{(k',m',n')\in\mathbb{N}^3\setminus\Gamma}\;
\sum_{(i_1,...,i_{l'})\in
S_{(k,m,n)\rightarrow(k',m',n')}}\pi^0_{i_1i_2}...\pi^0_{i_{l'-1}i_l'}
\\&\phantom{\leqslant C_1 \sum_{l'\geqslant
1}\sum_{(k',m',n')}}\times\sum_{\epsilon\in\mathbb{N}^3,\|\epsilon\|=1}
\mathbb{P}_{(k',m',n')+\epsilon}(T_{\Gamma_a}<T_{\Gamma_A})
\\&\leqslant 6C_1\sum_{l'\geqslant 1}\mathbb{P}_{(k,m,n)}(\mathcal{T}^{\nu}_{\Gamma}>l')
\quad\text{the latter sum being lower than $6$.}\\&
=6C_1\mathbb{E}_{(k,m,n)}(\mathcal{T}_{\Gamma}-1).\ea

From Proposition \ref{proptempsdarret},
$\mathbb{E}_{(k,m,n)}(\mathcal{T}_{\Gamma})<C_2(k+m+n)$ for a
constant $C_2$, which gives that $u((k,m,n),.)$ is differentiable with respect to
$\delta$ and that its derivative at $0$ $v(k,m,n)$ is sublinear.

\medskip
Now, identifying the first order terms in $\delta$ in
$\eqref{Deltau}$, we see that $v$ satisfies for all $(k,m,n)\in\mathbb{N}^3_{**}$:
\ben\label{Deltav}\left\{\begin{array}{l}(Lv)(k,m,n)=\frac{m(n-k)}{2N(N-1)} \quad\forall(k,m,n)\in\mathbb{N}^3_{**}\\
            v(k,0,0)=v(0,0,n)=0 \quad \text{if } k\geqslant2 \text{ and } n\geqslant2
           \end{array}\right.\een

It remains to prove that the system of Equations $\eqref{Deltavmin}$ admits a unique sub-polynomial solution. Let $h$ be a sub-polynomial solution of the equation $Lh=0$ such that $h(2,0,0)=h(0,0,2)=0$. Then $(h(\mathcal{Z}_{l}))_{l\in\mathbb{N}}$ is a $\mathcal{F}_l$-martingale. On $\Gamma_A$, $Lh(k,m,n)=0$ gives
\be bk(h(k+1,0,0)-h(k,0,0))=(dk+ck(k-1))(h(k,0,0)-h(k-1,0,0)) \quad\forall k\geqslant3\ee

which implies that $h\equiv 0$ on $\Gamma_A$ since $h$ is sub-polynomial and $h(2,0,0)=0$. Similarly, $h\equiv 0$ on $\Gamma_a$. Besides, there exists a positive integer $q$ such that
\be \sup_t\mathbb{E}_{k,m,n}(|h(Z_t)|^2)\leqslant
\sup_t\mathbb{E}_{k,m,n}(C|k_t+m_t+n_t|^{2q}).\ee

Moreover, from Proposition \ref{sup},
$\sup_t\mathbb{E}_{k,m,n}(|k_t+m_t+n_t|^{2q})<+\infty$ for all $(k,m,n)$ in $\mathbb{N}^3_{**}$. Then the martingale $(h(\mathcal{Z}_l))_{l\in\mathbb{N}}$ is uniformly
integrable. From Doob's stopping time theorem applied in the stopping time
$\mathcal{T}_{\Gamma}$, we then have $0=\mathbb{E}_{k,m,n}(h(\mathcal{Z}_{\mathcal{T}_{\Gamma}}))=h(k,m,n).$ \end{proof}

\medskip
Let us now state the following proposition whose proof will be the aim of Section \ref{sectionpreuve}.

\begin{prop}\label{propositionv}
For all $(k,m,n)$ such that $k+m+n\geqslant2$,
\be v(k,m,n)=(k-n)\left[\frac{m}{N}x_N+\frac{N^2-(k-n)^2}{N^2}y_N\right]\ee
where the sequence of vectors $(z_N)_{N\geqslant3}=\left(\begin{aligned}&x_N\\&y_N\end{aligned}\right)_{N\geqslant3}$
is the unique subpolynomial solution of the following system of equations:
\begin{eqnarray}B_Nz_{N+1}&=&C_Nz_N+D_Nz_{N-1}+f_N
\quad\quad\text{for all $N\geqslant4$}\label{equationz}\\
         B_3z_4&=&\tilde{C}_3z_3+\tilde{f}_3\label{CIz},
       \end{eqnarray}
with
\ba
B_N&:=\frac{b}{2(N-1)(N+1)}\left(\begin{array}{cc}1&
\frac{2N^2+4N-3}{N+1}\\2N^2-3&\frac{-3}{N+1}\end{array}\right),\\
C_N&:=(b+d+c(N-1))\left(\begin{array}{cc}0&\frac{1}{N}\\1&0\end{array}\right),\\
\tilde{C_3}&:=\left(\begin{array}{cc}0
&\frac{b+d+2c}{3}\\b+\frac{d+2c}{3}&-(d+2c)\end{array}\right)
=C_3-\left(\begin{array}{cc}0&0\\\frac{2}{3}(d+2c)&(d+2c)\end{array}\right),\\
D_N&:=-\frac{d+c(N-1)}{N-1}\left(\begin{array}{cc}0&\frac{N-3}{N-1}\\N-2&\frac{3}{N-1}\end{array}\right), \\
f_N&:=\left(\begin{array}{c}0\\\frac{-1}{2N(N-1)}\end{array}\right).\ea
\end{prop}

Note here that $v(k,m,n)=-v(n,m,k)$ and that the comparison between the proportions of genotypes $AA$ and $aa$ play a particular role in the value and sign of $v$.

\subsubsection{The dependence of $u$ in $\delta'$}\label{sectiondelta'}

For this section we set $\delta=0$, i.e. $a$ is a recessive allele, and deleterious when $\delta'>0$. As in the previous section (Proposition \ref{vSL}) $u^{0,.}_{k,m,n}:\delta'\mapsto u^{0,\delta'}$ is differentiable and $v'$ is the unique sublinear solution of the system

\begin{equation}\label{deltav'}\left\{\begin{array}{l}Lv'(k,m,n)=\frac{nY}{2N(N-1)}, \quad\quad\forall
(k,m,n)|k+m+n\geqslant2\\
v'(2,0,0)=v'(0,0,2)=0
\end{array}\right.
\end{equation}

where $Y=2k+m$ is the number of $A$ alleles in the population $(k,m,n)$.

The following proposition (proved in Subsection \ref{sectionpreuvev'}) gives a formula for $v'(k,m,n)$:

\begin{prop}\label{propositionv'}
\ben v'(k,m,n):=\frac{nY}{N}x_N+mx'_N+Y(2N-Y)\left(\frac{y'_N}{N}-\frac{Y}{2N^2}y_N\right)\een

where $x_N$ and $y_N$ are defined in Proposition \ref{propositionv}, and the sequence of vectors
$z'_N=\left(\begin{array}{c}x'_N\\y'_N\end{array}\right)$ is the unique subpolynomial solution of the following system of equations:
\begin{eqnarray}B'_Nz'_{N+1}&=&C'_Nz'_N+D'_Nz'_{N-1}+f'_N
\quad\quad\text{for all $N\geqslant3$}\label{equationz'}\\
         \tilde{B'}_2z'_3&=&\tilde{C'}_2z'_2+\tilde{f'}_2\label{CIz'},
       \end{eqnarray}
with
\ba
B'_N&:=\frac{b}{N-1}\left(\begin{array}{cc}2N^2-2N-1&\frac{-1}{N+1}\\
\frac{1}{2}&\frac{N^2+N-3/2}{N+1}\end{array}\right),\\
\tilde{B}'_2&:=\left(\begin{array}{cc}1&3\\3&\frac{13}{3}\end{array}\right),\\
C'_N&:=(bN+dN+cN(N-1))\left(\begin{array}{cc}2&0\\0&\frac{1}{N}\end{array}\right),\\
\tilde{C}'_2&:=\left(\begin{array}{cc}0&2\\2&3\end{array}\right),\\
D'_N&:=-(d+c(N-1))\left(\begin{array}{cc}2N-2&\frac{2}{N-1}\\0&\frac{N-2}{N-1}\end{array}\right), \\
f'_N&:=\left(\begin{array}{c}\frac{b}{N-1}(2N-1)\frac{y_{N+1}}{2(N+1)^2}-(d+c(N-1))(4N+2)
\frac{y_{N-1}}{2(N-1)^2}\\\\\begin{aligned}\frac{b}{N-1}
&\left(2N^3+3N^2-4N-\frac{3}{2}\right)\frac{y_{N+1}}{2(N+1)^2}\\&-(bN+dN+cN(N-1))(2N-1)\frac{y_N}{2N^2}
\\&\phantom{+bN}+(d+c(N-1))(2N^2-7N+8)
\frac{y_{N-1}}{2(N-1)^2}\end{aligned}\end{array}\right),\\
\tilde{f}'_2&:=\left(\begin{array}{c}x_2-y_2-x_3+\frac{3}{2}y_3\\\frac{19}{6}y_3-\frac{9}{4}y_2
\end{array}\right).\ea
\end{prop}

We now prove Propositions \ref{propositionv} and \ref{propositionv'}. In both cases, the proof is shared in two parts: we first prove the result when the fecundity $b$ is small
enough compared to the competition parameter $c$, and then we
generalize the result to all possible demographic parameters $b$, $d$, and $c$.

\section{Proofs of Propositions \ref{propositionv} and \ref{propositionv'}}\label{sectionpreuve}
\subsection{Proof of Proposition \ref{propositionv} for small $b$}\label{sectionpreuvebpetit}

To begin with, straightforward calculations give the following lemma:

\begin{lem}
\begin{description}
\item[$(i)$] If $\eqref{formulev}$ is true, then $v$ satisfies $\eqref{Deltavmin}$ if and only if
$(z_N)_{N\geqslant3}$ satisfies $\eqref{equationz}$, $\eqref{CIz}$ and $x_2+\frac{3}{2}y_2=\frac{4}{3}x_3+2y_3$.
\item[$(ii)$] $(v(k,m,n))_{(k,m,n)\in\mathbb{N}^3_{**}}$ is sublinear if and only if $(z_N)_{N\geqslant3}$ is bounded.
\end{description}
\end{lem}

Notice that $z_2$ can not be computed; indeed $v(1,1,0)=-v(0,1,1)=\frac{1}{2}x_2+\frac{3}{4}y_2$ and $v(k,m,n)=0$ elsewhere.

\bigskip
We then only have to prove that there exists a bounded solution
$(z_N)_{N\geqslant3}$ to the system of Equations $\eqref{equationz}$ and $\eqref{CIz}$.
Notice that if $z_3$ is fixed then for all $N$, $z_N$ is fixed, recursively. Finding a
bounded solution of this system of equations is then equivalent to
finding an initial condition $z$ (necessarily unique by Proposition
\ref{vSL}) such that if $z_3=z$ then $(z_N)_{N\geqslant3}$ is
bounded.

\subsubsection{The one-order recurrence relationship satisfied by $(z_N)_N$}\label{section1drecurrence}

We change the two-order
recurrence system of Equations $\eqref{equationz}$ and $\eqref{CIz}$ into a one-order
recurrence relationship, so that we can express easily $z_N$ as a
function of $z_3$ and conversely. We easily find that $z_N$
satisfies the following recurrence relationship:
\ben \label{recsimple}
B_Nz_{N+1}=(C_N+K_N)z_N+\sum_{k=3}^{N}(-1)^{N-k}E(k,N)f_k \quad\text{for all $N\geqslant3$.}\een
More precisely, $\eqref{recsimple}$ is satisfied for $N=3$ if
$K_3=\tilde{C_3}-C_3$ and $E(3,3)=I_2$. Moreover, if it is true for a given $N\geqslant3$ then it is true
for $N+1$ as long as $K_{N+1}=D_{N+1}(C_{N}+K_{N})^{-1}B_{N}$,
$E(N+1,N+1)=I_2$ and $E(k,N+1)=D_{N+1}(C_{N}+K_{N})^{-1}E(k,N)$ for
all $k\in\mathbb{[}3,N\mathbb{]}$. Then the recurrence
relationship $\eqref{recsimple}$ is satisfied for every $N$ as soon as
we can define two sequences of matrices $(K_N)_{N\geqslant3}$ and
$(E_N)_{N\geqslant3}$ such that:

\be \left\{\begin{array}{l}K_N=D_N(C_{N-1}+K_{N-1})^{-1}B_{N-1}\quad\quad\quad\forall N\geqslant4\\
         K_3=\tilde{C_3}-C_3\\
         E(k,N)=D_N(C_{N-1}+K_{N-1})^{-1}E(k,N-1)\quad\quad\quad\forall k\in\mathbb{[}3,N-1\mathbb{]}\\
         E(k,k)=I_2 \quad \quad\quad\forall k\geqslant3
         \end{array}\right.\ee

We then have to prove recursively that $F_N:=K_N+C_N$ is
invertible for all $N\geqslant3$. We first prove it when
$c$ is large enough compared to $b$.

\subsubsection{Proof of the invertibility of $K_N+C_N$}

Let us define
\ba V_N&:=\left(\begin{array}{cc}0&\frac{1}{N}\\1&0\end{array}\right).\ea

Then $F_N=(b+d+c(N-1))V_N+K_N.$ We now define the matrix
$G_N:=V_N+\frac{1}{b}K_N$. Then
\ba F_N&=(d+c(N-1))V_N+bG_N\\&=(d+c(N-1))V_N\left(I_2+\frac{b}{d+c(N-1)}V_N^{-1}G_N\right).\ea

Using the matricial norm
$\|M\|=\sup_{i\in\{1,2\}}(|M_{i,1}|+|M_{i,2}|)$, note that $\|V_N^{-1}\|=N$.

\begin{lem} \label{FG}
If $b\leqslant\frac{c}{24}$, then $F_N$ is invertible and $\|G_N\|\leqslant 9$ for all $N\geqslant4$.
\end{lem}

This result will be generalized in Subsection \ref{sectiongeneralisation} to all possible parameters $b$, $d$, and $c$.

\bigskip
\begin{proof}{(of Lemma \ref{FG})}
We prove it recursively. For $N=4$, we can compute the norm of $G_4$. Indeed we have:
\be G_4=V_4+\frac{1}{b}D_4\tilde{C}_4^{-1}B_3,\ee

which gives us:
\be
G_4=\left(\begin{array}{lr}-\frac{d+3c}{48(b+d+2c)}&\frac{1}{4}-\frac{9(d+3c)}{64(b+d+2c)}\\&\\&\\1
-\frac{d+3c}{16(b+d+2c)}&-\frac{27}{64}
\frac{(d+3c)(b+2(d+2c)+\frac{d+2c}{3})}{(b+d+2c)(b+\frac{d+2c}{3})}\\-\frac{10(d+3c)}{16(b+\frac{d+2c}{3})}
-\frac{(d+2c)(d+3c)}{8(b+d+2c)(b+\frac{d+2c}{3})}&+\frac{d+3c}{32(b+\frac{d+2c}{3})}\end{array}\right).\ee

So
\ba \|G_4\|&\leqslant
\sup\left\{\frac{d+3c}{d+2c}\left(\frac{1}{48}+\frac{1}{4}+\frac{9}{64}\right),\frac{d+3c}{d+2c}
\left(1+\frac{1}{16}+\frac{30}{16}+\frac{3}{8}+\frac{1}{64}\right)\right\}\\&=\frac{d+3c}{d+2c}\frac{212}{64}\leqslant
\frac{212}{64}\frac{3}{2}\leqslant 9.\ea

For all $N$, the
invertibility of the matrix $F_N$ is a consequence of $\|G_N\|\leqslant 9$. Indeed, if $\|G_N\|\leqslant 9$, then as long as $b<\frac{c}{12}$,
\be \left\|\frac{bV_N^{-1}G_N}{d+c(N-1)}\right\|\leqslant
\frac{9bN}{d+c(N-1)}<1.\ee

In this case,
$I_2+\frac{bV_N^{-1}G_N}{d+c(N-1)}$ is invertible, and so is $F_N$. Now let us assume that $\|G_N\|\leqslant 9$ for a given $N\geqslant4$ and
let us prove that $\|G_{N+1}\|\leqslant 9$. If $\|G_N\|\leqslant 9$, then $F_N$ is invertible and we can write $G_{N+1}=V_{N+1}+\frac{1}{b}D_{N+1}F_N^{-1}B_N.$ Hence
\be
G_{N+1}=V_{N+1}+D_{N+1}\left(I_2+\frac{bV_N^{-1}G_N}{d+c(N-1)}\right)^{-1}\frac{V_N^{-1}}{d+c(N-1)}\frac{B_N}{b}.\ee

Moreover, as long as $b\leqslant\frac{c}{24}$,
\be\left\|\left(I_2+\frac{bV_N^{-1}G_N}{d+c(N-1)}\right)^{-1}\right\|\leqslant\frac{1}{1-
\left\|\frac{bV_N^{-1}G_N}{d+c(N-1)}\right\|}\leqslant\frac{1}{1-\frac{9bN}{d+c(N-1)}}\leqslant2.\ee

Finally, for all $N\geqslant4$, $\|D_{N+1}\|\leqslant d+cN$ and $\|V_N^{-1}B_N\|\leqslant 3b$ which implies
\be \|G_{N+1}\|\leqslant1+6\left(1+\frac{c}{d+3c}\right)\leqslant
9.\ee\end{proof}

\medskip
As long as $b\leqslant c/24$, Equation $\eqref{recsimple}$ is satisfied,
which allows us to express easily $z_N$ as a function of $z_3$ for
all $N\geqslant3$. We now prove that there exists a real number $z$ such that if $z_3=z$ then $(z_N)$ is bounded.

\subsubsection{Boundedness of $z$}\label{sectionzborne}

Let us assume here that $b<c/24$, so that we can use the previous results. Setting
\be M_N:=B_N^{-1}(C_N+K_N),\quad\text{ and }\quad g_N:=\sum_{k=3}^{N}(-1)^{N-k}B_N^{-1}E(k,N)f_k,\ee

we get
\ben \label{zN} z_{N+1}=M_NM_{N-1}...M_3(z_3+\sum_{l=3}^NM_3^{-1}..M_l^{-1}g_l)=P_N\left(z_3+\sum_{l=3}^NP_l^{-1}g_l\right)\een

if $P_N=M_NM_{N-1}...M_3$. To obtain the behaviour of $(z_N)$, we then study $P_N$ and $g_N$.

\begin{lem}\label{LN}
 $\|M_N^{-1}\|\leqslant \frac{2b}{cN}$ if $N$ is large enough.
\end{lem}

\begin{proof}{(of Lemma \ref{LN})}  We previously proved (Lemme \ref{FG}) that for all $N\geqslant3$, $\|G_N\|\leqslant 9$, with $G_N=V_N+\frac{K_N}{b}$. Then for all $N\geqslant3$, $\|K_N\|\leqslant 10b$. So if $b<\frac{c}{24}$, we have
\ben\label{norme}\|K_N\|<c/2\een

for all $N\geqslant3$.
Besides, the equation $K_{N+1}=D_{N+1}(C_N+K_N)^{-1}B_N$ can be detailed, and using Equation $\eqref{norme}$, we obtain
that \ben\label{equivalentK}K_{N+1}=-b\left(\begin{array}{cc}\frac{1}{2N^2}+O\left(\frac{1}{N^3}\right)&\frac{1}{N}
+O\left(\frac{1}{N^3}\right)\\1
+O\left(\frac{1}{N^2}\right)&\frac{3}{N^2}+O\left(\frac{1}{N^3}\right)\end{array}\right).\een
Next, \be
D_{N+1}^{-1}=\frac{N^2}{(d+cN)(N-2)(N-1)}\left(\begin{array}{cc}\frac{3}{N}&-\frac{N-2}{N}
\\-(N-1)&0\end{array}\right).\ee

We deduce from this that
\ben\label{LN-1}M_N^{-1}=D_{N+1}^{-1}K_{N+1}=\frac{b}{c}\left(\begin{array}{cc}\frac{1}{N}+O\left(\frac{1}{N^2
}\right)&O\left(\frac{1}{N^3}\right)
\\\frac{1}{2N^2}+O\left(\frac{1}{N^3}\right)&\frac{1}{N}
+O\left(\frac{1}{N^2}\right)\end{array}\right).\een\end{proof}

\medskip
Notice that if $N$ is large enough
\ben\label{LNrem}\|M_N^{-1}M_{N+1}^{-1}\|\leqslant\frac{4b^2}{c^2N^2}.\een

Besides, we have the following lemma for $(g_N)_{N}$:

\begin{lem}\label{g} $g$ satisfies
\ben\label{developpementg}g_N=C+\frac{C'}{N}+o\left(\frac{1}{N}\right)\een
\end{lem}

\begin{proof}{(of lemma \ref{g})} From $g_N:=\sum_{k=3}^{N}(-1)^{N-k}B_N^{-1}E(k,N)f_k$ we deduce
\ben \label{recurrenceg} g_{N+1}=-B_{N+1}^{-1}K_{N+1}g_N+B_{N+1}^{-1}f_{N+1}\een
Moreover,
\ben \label{BN-1}
B_N^{-1}=\frac{1}{b}\frac{2(N-1)(N+1)^2}{3+(2N^2+4N-3)(2N^2-3)}\left(\begin{array}{cc}\frac{3}{N+1}&\frac{2N^2
+4N-3}{N+1}\\2N^2-3&-1
\end{array}\right)\een

and Equation $\eqref{equivalentK}$ yields
\ben \label{BNKN}
-B_{N}^{-1}K_N=\left(\begin{array}{cc}1+O\left(\frac{1}{N^2}\right)&\frac{3}{N^2}+O\left(\frac{1}{N^3}\right)
\\O\left(\frac{1}{N^2}\right)&1+O\left(\frac{1}{N^2}\right)\end{array}\right).\een

Equations $\eqref{recurrenceg}$ and $\eqref{BN-1}$ and $\eqref{BNKN}$ give us the result.\end{proof}

\medskip
Finally, we get interested in $\sum P_l^{-1}g_l$. Let us recall that $P_l=M_lM_{l-1}...M_3$.
\be
\sum_{l=3}^N\|P_l^{-1}g_l\|\leqslant\sum_{l=3}^{N}\|P_l^{-1}\|\|g_l\|.\ee

From $\eqref{developpementg}$ and Lemma
\ref{LN}, $(g_l)_{l\geqslant3}$ is bounded and there exists a constant $C_2$ such that
$\|M_N^{-1}\|\leqslant\frac{C_2}{N}$ when $N$ is large enough. Then
$\sum_{l=3}^NP_l^{-1}g_l$ converges and we define its limit
\ben \label{z}z=\sum_{l=3}^{\infty} P_l^{-1}g_l.\een
The quantity $z$ will be the initial condition, we need to obtain a bounded solution to $\eqref{equationz}$ and $\eqref{CIz}$ as is proved now:

\begin{lem} The sequence $(z_N)_{N\geqslant3}$ satisfying $\eqref{equationz}$ and $\eqref{CIz}$, and such that $z_3=-z$ (where $z$ has been defined in $\eqref{z}$), is bounded.
\end{lem}

\begin{proof} From $\eqref{zN}$
\ben\begin{aligned}\label{zNfinal}z_{N+1}&=
-P_N\times\left(\sum_{l=N+1}^{\infty}P_l^{-1}g_l\right)
=-\sum_{l=N+1}^{\infty}M_{N+1}^{-1}M_{N+2}^{-1}...M_l^{-1}g_l
\\&=-M_{N+1}^{-1}g_{N+1}-M_{N+1}^{-1}M_{N+2}^{-1}\sum_{l=N+2}^{\infty}(M_{N+3}^{-1}..M_{l-2}^{-1})(M_{l-1}^{-1}M_{l}^{-1})g_l.
\end{aligned}\een
By Lemmas \ref{g} and \ref{LN} and Equation $\eqref{LNrem}$, if
$N$ is large enough, there exists a constant $C$ independent from $b$ such that
\ben\label{zN1surN} \|z_N\|\leqslant C\frac{2b}{cN}.\een\end{proof}

\medskip
Proposition \ref{propositionv} is now proved for small $b$. In the next subsection we generalize this result to any $b$.

\subsection{Generalization to all possible values of $b$}\label{sectiongeneralisation}

\begin{thm}\label{vanalytique}
For all $(k,m,n)$ such that $k+m+n\geqslant2$,
$v(k,m,n)$ is an analytic function of $b$ on
$\mathbb{R}^{+*}$.
\end{thm}

\begin{cor}\label{corollairefinal}
For all demographic parameters $b>0$, $d$, and $c>0$, Proposition \ref{propositionv} is true.
\end{cor}

\begin{proof}{(of Corollary \ref{corollairefinal}.)} From the end of Section \ref{sectionzborne},
there exists a constant $K>0$ such that if $b<Kc$, $\eqref{formulev}$ is true, which gives
\ba
y_N&=\frac{N^2}{4(N-2)(N-1)}v(N-1,0,1),\\
x_N&=\frac{N}{N-1}\left[v(N-1,1,0)-\frac{2N-1}{4(N-2)}v(N-1,0,1)\right].
\ea

As long as $b<Kc$ we then have
\ban \label{equalityv}v(k,m,n)&=\frac{m(k-n)}{N-1}\left[v(N-1,1,0)-\frac{2N-1}{4(N-2)}v(N-1,0,1)\right]\\&
+(k-n)\frac{N^2-(k-n)^2}{4(N-2)(N-1)}v(N-1,0,1).\ean

Now from Theorem \ref{vanalytique}, for all $(k,m,n)$ in $\mathbb{N}^3_{**}$,
$v(k,m,n)$ is an analytic function of $b$ on
$\mathbb{R}^{+*}$. The equality $\eqref{equalityv}$ of two
analytic functions on $]0,Kc[$ extends on
$\mathbb{R}^{+*}$.\end{proof}

\medskip
Before proving Theorem \ref{vanalytique}, we prove

\begin{lem}For every $(k,m,n)$ in $\mathbb{N}^3_{**}$, there exists a strictly positive real number $\rho$ such that
$\mathbb{E}_{k,m,n}((1+\rho)^{\mathcal{T}_{\Gamma}})<\infty.$
\end{lem}

\begin{proof} We define the random number $L\in\mathbb{N}$ of return of $\mathcal{Z}$ in
$\{N=2\}$ before reaching $\Gamma$, and $\mathcal{T}^{(i)}_2$ the
$i$-th time of return of $\mathcal{Z}$ in $\{N=2\}$ ($T^{(0)}_2=0$
and $\mathcal{T}^{(1)}_2=\mathcal{T}_{\{2\}}$). \ba
\mathbb{E}_{k,m,n}((1+\rho)^{\mathcal{T}_{\Gamma}})&\leqslant\sum_{l=0}^{\infty}\mathbb{E}_{k,m,n}
\left((1+\rho)^{\mathcal{T}^{(l+1)}_2}\mathbf{1}_{L=l}\right)\quad
\text{ as } \mathcal{T}_{\Gamma}\mathbf{1}_{L=l}\leqslant
\mathcal{T}^{(l+1)}_2\mathbf{1}_{L=l}\\&
=\sum_{l=0}^{\infty}\sum_{\substack{(k',m',n')\notin\Gamma|\\k'+m'+n'=2
\text{
or}\\(k',m',n')=(k,m,n)}}\mathbb{E}_{k,m,n}\left((1+\rho)^{\mathcal{T}^{(l+1)}_{2}}
\mathbf{1}_{Z_{\mathcal{T}^{(l)}_2}=(k',m',n')}\mathbf{1}_{L=l}\right)
\\&\leqslant\!\!\!\!\!\!\!\max_{\substack{(k',m',n')\notin\Gamma|\\k'+m'+n'=2 \text{ or}\\(k',m',n')=(k,m,n)}}\!\!\!\!\!\!
\mathbb{E}_{k',m',n'}\left((1+\rho)^{\mathcal{T}_{\{2\}}}\right)
\times\sum_{l=0}^{\infty}\mathbb{E}_{k,m,n}\left(
(1+\rho)^{\mathcal{T}^{(l)}_2}\mathbf{1}_{L\geqslant l}\right),\ea
by strong Markov property in $\mathcal{T}^{(l)}_2$. We now define
$$S=\max_{\substack{(k',m',n')|\\k'+m'+n'=2 \text{
or}\\(k',m',n')=(k,m,n)}}\mathbb{E}_{k',m',n'}\left(
(1+\rho)^{\mathcal{T}_{\{2\}}}\mathbf{1}_{L\geqslant1}\right)$$ and
prove that for every $l$, $$\mathbb{E}_{k,m,n}\left(
(1+\rho)^{\mathcal{T}^{(l)}_2}\mathbf{1}_{L\geqslant
l}\right)\leqslant S^l.$$ The result is obviously true for $l=0$ and
is proved recursively for every $l$ by using strong Markov property
in $T^{(l-1)}_2$ as previously. Now from Proposition
\ref{proptemps}, for every $(k,m,n)$ there exists $\rho>0$ such that
$\mathbb{E}_{k,m,n}((1+\rho)^{\mathcal{T}_{\{2\}}})<\infty$. Then by
the Dominated Convergence Theorem,
$\mathbb{E}_{k,m,n}((1+\rho)^{\mathcal{T}_{\{2\}}}\mathbf{1}_{L\geqslant1})\underset{\rho\rightarrow0}
{\longrightarrow}\mathbb{P}_{k,m,n}(L\geqslant1)<1$. Hence there
exists $\rho_0$ such that if $\rho<\rho_0$, $S<1$ and then
$\mathbb{E}_{k,m,n}((1+\rho)^{\mathcal{T}_{\Gamma}})<\infty$.
\end{proof}

\begin{proof}{(of Theorem \ref{vanalytique})} We
need to study the dependence of the probability $u$ in the fecundity
parameter $b$, so we denote by $u((k,m,n),\delta,b)$ the fixation
probability of allele $a$ when $Z_0=(k,m,n)$ and $v((k,m,n),b)$ its
derivative with respect to $\delta$. If
$u((k,m,n),.,.)$ is an analytic function of $(b,\delta)$ on
$\mathbb{R}^{+*}\times\mathbb{R}$, then $v((k,m,n),.)$ is an
analytic function of $b$ on $\mathbb{R}^{+*}$. Now,
\be u((k,m,n),\delta,b)=\sum_{l\geqslant1}\sum_{(i_1,..i_l)\in
S_{(k,m,n)\rightarrow\Gamma_a}}
\pi^{\delta,b}_{i_1i_2}...\pi^{\delta,b}_{i_{l-1}i_l},\ee where
$\pi^{\delta,b}_{i_ki_{k+1}}$ is the transition probability from
state $i_k$ to state $i_{k+1}$ and an analytic function of
$(b,\delta)$ on $\mathbb{R}^{+*}\times\mathbb{R}$. $u$ is then the
simple limit of analytic functions on
$\mathbb{R}^{+*}\times\mathbb{R}$. By $(9.13.1)$
and $(9.13.2)$ of \citet{Dieudonne}, a sequence of analytic functions $(f_n)_{n}$
defined on an open set $S$ of $\mathbb{C}$ which converges
simply towards a function $f$ on $S$ , is proved to converge uniformly on every compact subset of $S$ as long as $\{f_n,n\in\mathbb{N}\}$ is relatively compact.
We extend the functions $\pi^{\delta,b}_{i_ki_{k+1}}$ on the open set $E_1^{\beta}\times
E_{2}^{\beta}$ where $\beta\in\mathbb{R^{+*}}$ and
\be E_{1}^{\beta}=\{z\in\mathbb{C}|Re(z)>0, |Im(z)|<\beta
Re(z)\},\ee
\be E_{2}^{\beta}=\{z\in\mathbb{C}||Re(z)|<d/2,
|Im(z)|<\beta(d-|Re(z)|+2c)\}.\ee We set $b=b_r+ib_i\in
E_{1}^{\beta}$, $\delta=\delta_r+i\delta_i\in E_{2}^{\beta}$ and denote by
$P^{b,\delta}_{(k,m,n)(k',m',n')}$ the analytic extension of
$\pi^{b,\delta}_{(k,m,n)(k',m',n')}$ on $E_1^{\beta}\times
E_{2}^{\beta}$.
For all $(b,\delta)\in E_{1}^{\beta}\times E_{2}^{\beta}$ and for
all $(k,m,n)$ and $(k',m',n')$ neighbors in $\mathbb{N}^3$:
\be \left|P^{b,\delta}_{(k,m,n)(k',m',n')}\right|\leqslant
\sqrt{1+\beta^2}\;P^{b_r,\delta_r}_{(k,m,n)(k',m',n')}=\sqrt{1+\beta^2}\;\pi^{b_r,\delta_r}_{(k,m,n)(k',m',n')}
.\ee
Indeed, let us make the computation if $(k',m',n')=(k,m-1,n)$,
\ba \left|P^{b,\delta}_{(k,m,n)(k,m-1,n)}\right|&=
\left|\frac{(d+\delta+c(N-1))m}{bN+dN+\delta m+cN(N-1)}\right|\\&\leqslant
\frac{|(d+\delta+c(N-1))m|}{Re(bN+dN+\delta m+cN(N-1))}\\&=
\frac{\sqrt{(d+\delta_r+c(N-1))^2m^2+\delta_i^2m^2}}{b_rN+dN+\delta_r m+cN(N-1)}\\&\leqslant
\frac{(d+\delta_r+c(N-1))m\left(\sqrt{1+\frac{\delta_i^2}{(d+\delta_r+c(N-1))^2}}\right)}{b_rN+dN+\delta_r m+cN(N-1)}\\&
\leqslant\frac{(d+\delta_r+c(N-1))m\sqrt{1+\beta^2}}{b_rN+dN+\delta_r m+cN(N-1)}
\quad\quad\text{since $\delta\in E_{2}^{\beta}$}\\&=\sqrt{1+\beta^2}\;P^{b_r,\delta_r}_{(k,m,n)(k,m-1,n)}
\ea
Computations are similar for other possible transitions. Then since $\sqrt{1+\beta^2}\leqslant1+\beta^2$,
\ba \sum_{l\geqslant1}^L\sum_{(i_1,..i_l)\in
S_{(k,m,n)\rightarrow\Gamma_a}}\!\!\!\!\!\!\!|P^{\delta,b}_{i_1i_2}...P^{\delta,b}_{i_{l-1}i_l}|&\leqslant
\sum_{l\geqslant1}^L(1+\beta^2)^l\!\!\!\!\!\!\sum_{(i_1,..i_l)\in
S_{(k,m,n)\rightarrow\Gamma_a}}\!\!\!\!\!\pi^{\delta_r,b_r}_{i_1i_2}...\pi^{\delta_r,b_r}_{i_{l-1}i_l}\\&
\leqslant\sum_{l\geqslant1}^L(1+\beta^2)^l
\mathbb{P}_{k,m,n}(\mathcal{T}_{\Gamma_a}=l)\\&\leqslant\mathbb{E}_{k,m,n}((1+\beta^2)^{\mathcal{T}_{\Gamma_a}}
\mathbf{1}_{\mathcal{T}_{\Gamma_a}<\infty})
\leqslant\mathbb{E}_{k,m,n}((1+\beta^2)^{\mathcal{T}_{\Gamma}})\ea
since, if $\mathcal{T}_{\Gamma_a}<\infty$,
$\mathcal{T}_{\Gamma_a}=\mathcal{T}_{\Gamma}$.\end{proof}

\medskip
In the following subsection, we establish some properties of the
derivative $v(k,m,n)$.
\subsection{Boundedness and sign of $v$}
\begin{prop}\label{propsignev}
\begin{description}
 \item[$(i)$] For all demographic parameters $b$, $d$ and $c$, $v$ is a bounded function of $(k,m,n)$.
\item[$(ii)$] $v_{k,m,n}=\mathbb{E}_{(k,m,n)}\left[\int_0^{T}Lv(Z_t)dt\right]\geqslant0$ where $T=\inf\{t,k_t=n_t \text{ or } m_t=n_t=0\}.$
\item[$(iii)$] $v(k,m,n)$ has the same sign than $k-n$.
\end{description}
\end{prop}
\begin{proof} $(i)$ is a consequence of Equation $\eqref{zN1surN}$
and $(iii)$ is a consequence of $(ii)$. For $(ii)$, by Proposition
\ref{propositionv}, it suffices to prove the result when $k>n$. The
function $v$ being bounded in $(k,m,n)$ (by $(i)$), Dynkin's formula stopped at
the stopping time $T$ gives us that
\be
\mathbb{E}_{k,m,n}[v(Z_{T})]=v(k,m,n)-\mathbb{E}_{(k,m,n)}\left[\int_0^{T}Lv(Z_t)dt\right].\ee
Using that $v(Z_T)=0$ (from Proposition \ref{propositionv}), we get the result.\end{proof}

\medskip
Notice that the sign of $\delta$ is not sufficient to know whether
the allele $a$ has a larger fixation probability than a neutral
allele, or not. This property depends on the initial genetic
repartition of the population: if there are more alleles $A$ (resp.
$a$) initially, then allele $a$ has a lower fixation probability
than a neutral allele if and only if $\delta>0$ (resp. $\delta<0$).
In Section \ref{sectionmutationnelle}, we will get interested in the particular case
where the allele $a$ is a mutant appearing in the population. In
this case, at mutation time, there is only one individual with
genotype $Aa$ and no individual with genotype $aa$, then the
population starts from a state of the form $(k,1,0)$. The fixation
probability of allele $a$ is then:
\be u((k,1,0),\delta)=\frac{1}{2(k+1)}
-\delta\left(\frac{k}{k+1}x_{k+1}+\frac{k(2k+1)}{(k+1)^2}y_{k+1}\right)\\+o(\delta)\ee
\subsection{Proof of Proposition \ref{propositionv'}}\label{sectionpreuvev'}
As in computations for $v$, Proposition \ref{propositionv'} is true
if we can find a bounded sequence $(z'_N)_{N\geqslant2}$ which is
solution of $\eqref{equationz'}$ and $\eqref{CIz'}$. To prove this,
we use a similar proof as for $\delta'=0$ (Section
\ref{sectionpreuvebpetit}). Setting \ba h_k&=f'_k \quad\quad\forall
k\geqslant4\\h_3&=f'_3-D'_3\tilde{C}'^{-1}_2\tilde{f}'_2,\ea we easily
obtain that for all $N\geqslant3$:
\ben B'_Nz'_{N+1}=(C'_N+K'_N)z'_N+\sum_{k=3}^N(-1)^kE'(N,k)h_k\een
with
\ba
K'_3&=D'_3\tilde{C}'^{-1}_2\tilde{B}'_2\\K'_N&=D'_N(C'_{N-1}+K'_{N-1})^{-1}B'_{N-1}
\quad\quad\forall N\geqslant4\\E'(k,k)&=I_2\quad\quad\forall
k\geqslant3\\E'(N,k)&=D'_N(C'_{N-1}+K'_{N-1})^{-1}E'(N-1,k)
\\&=K'_NB'^{-1}_{N-1}E'(N-1,k)\quad\quad\forall N\geqslant k+1\ea
Notice here that the detailed computation of $h_3$ shows that $h_3$
does not depend on $x_2$ and $y_2$ (which are not known) but only on
$x_2+\frac{3}{2}y_2$.
The only difficulty in adapting the proof of Section
\ref{sectionpreuvebpetit} is when proving that there exists a
constant $C$ such that for all $N$,
$\|B'^{-1}_Nh_N\|\leqslant\frac{C}{N^2}$. Note that we have
\ba
B'^{-1}_N&=\frac{N-1}{b}\frac{N+1}{(2N^2-2N-1)(N^2+N-3/2)+1/2}\\&\times\left(\begin{array}{cc}\frac{N^2+N-3/2}
{N+1}&\frac{1}{N+1}
\\-\frac{1}{2}&2N^2-2N-1\end{array}\right).\ea
From Equations $\eqref{LN-1}$, $\eqref{developpementg}$ and
$\eqref{zNfinal}$,
\be
y_N=\frac{C_1}{N}+\frac{C_2}{N^2}+O\left(\frac{1}{N^3}\right).\ee
Then
\be
\left\|\left(\begin{array}{cc}\frac{N^2+N-3/2}{N+1}&\frac{1}{N+1}\\-\frac{1}{2}&2N^2-2N-1\end{array}\right)h_N
\right\|=O(1)
\quad\text{and}\quad\|B'^{-1}_Nh_N\|=O\left(\frac{1}{N^2}\right).\ee
We now know that if the birth parameter $b$ is small enough compared
to $c$, then $v'$ is effectively defined as in Formula
$\eqref{formulev'}$. To generalize this result to all possible
values of parameters $b$ and $c$, we adapt the proof of Theorem
\ref{vanalytique} and Corollary \ref{corollairefinal} to $\delta'$,
without any difficulty. Note here that for all demographic
parameters, $v'$ is a positive bounded function of $(k,m,n)$.
\subsection{Proof of the analyticity of
$u(k,m,n)$}\label{sectiondifferentiabilite} To conclude these
results, we now prove that $u((k,m,n),\delta,\delta')$ is an
analytic function of $(\delta,\delta')$ in the neighborhood of
$(0,0)$.
\begin{proof} We use analytic extension arguments as in the proof
of Theorem \ref{vanalytique}. Here $\delta$ and $\delta'$ are
complex numbers, denoted by $\delta=\delta_r+i\delta_i$ and
$\delta'=\delta'_r+i\delta'_i$. We take $(\delta,\delta')\in
(E^{\beta})^2$ with $E^{\beta}=\{z\in\mathbb{C}||Re(z)|<d/2,
|Im(z)|<\beta(d-|Re(z)|+2c)\}$, and denote by
$\pi^{\delta,\delta'}_{(k,m,n)(k',m',n')}$ the transition
probability for $Z$ from $(k,m,n)$ to one of its neighbor
$(k',m',n')$ and $P^{\delta,\delta'}_{(k,m,n)(k',m',n')}$ the
analytic continuation of $\pi^{\delta,\delta'}_{(k,m,n)(k',m',n')}$
on $(E^{\beta})^2$. Then, \ba
\left|P^{\delta,\delta'}_{(k,m,n)(k',m',n')}\right|&\leqslant
(1+\beta^2)P^{\delta_r,\delta'_r}_{(k,m,n)(k',m',n')}\\&=(1+\beta^2)
\pi^{\delta_r,\delta'_r}_{(k,m,n)(k',m',n')}.\ea Indeed, it is
proved by making the computation for all possible transitions as in
the proof of Theorem \ref{vanalytique} and the conclusion follows
similarly.\end{proof}

\medskip
Theorem \ref{maintheorem} is now proved.

\section{Mutational scale: convergence and extinction vortex}\label{sectionmutationnelle}
Understanding and quantifying the extinction risk of a population is
a very important issue, in particular within the framework of
species conservation \citet{GilpinSoule1986}. We now get interested
in a phenomenon called ``mutational meltdown'' \citet{Lynch1995}:
within small populations, inbreeding favors the fixation of
deleterious alleles that would disappear in an infinite size
population \citet{CrowKimura1970,ChampagnatMeleard2011,Metzetal1996}. This
phenomenon is then characterized by more and more frequent fixations
of deleterious alleles, which creates an extinction vortex and leads
to a rapid extinction of the population
\citet{Lande1994, GilpinSoule1986}. We wish now to observe this
acceleration of mutation fixations. To this end, we introduce
mutations in our model, and consider a different time scale.
\subsection{General model}\label{sectionmodele}
As introduced in Section \ref{sectionecologique}, each individual is
now characterized by its genotype
$x\in\mathbf{G}:=\{\{\mathcal{A},\mathcal{C},\mathcal{G},\mathcal{T}\}^{G}\}^2$. Now every DNA strand can now mutate during the individual lifetime,
at rate $\mu_K:=\mu/K$. $K$ is a scaling parameter that will go to
infinity, following a rare mutation hypothesis, which is usual
in evolutionary genetics \citet{Lande1994,Champagnat2006}. For
every $a,a'\in
\{\mathcal{A},\mathcal{C},\mathcal{G},\mathcal{T}\}^{G}$, we define
the probability $M(a,a')$ that a DNA strand $a$ mutates to $a'$
knowing that $a$ mutates. The population can then be represented at
time $t$ by
\be Z^K:t\mapsto\sum_{i=1}^{N^K_t}\delta_{x^{i,K}_t},\ee where
$N^K_t$ is the size of population $Z^K$ at time $t$ and $x^{i,K}_t$
is the genotype of the $i$-th individual in population $Z^K$ at time
$t$. $Z^K_t$ belongs to the discrete space:
$$E=\left\{\sum_{i=1}^{N}\delta_{x_{i}}, N\in\mathbb{N}, x_i\in\mathbf{G} \;\forall i\right\},$$ where $E$ is
equipped with its discrete topology and the norm
$r(\mu,\nu)=\sum_{x\in\mathbf{G}}|\mu(x)-\nu(x)|$. We denote by
$\mathbb{D}([0,\infty),E)$ the Skhorohod space of left limited right
continuous functions from $\mathbb{R}^+$ to $E$, endowed with the
Skhorohod topology. We denote by $b(x,Z)$ the birth rate of an individual with genotype $x$ in the
population $Z$, and assume that there exists a constant $\overline{C}$ such that
for every $Z$ with size $N$,
$\displaystyle{\sum_{x\in\mathbf{G}}b(x,Z)}\leqslant\overline{C}N.$
As in Section \ref{sectionecologique}, individuals can die either
naturally, or due to competition with other individuals, and when
the population size reaches $2$ we assume that no death can occur.
We denote by $d(x,Z)$ the death rate of a given individual with
genotype $x$ in the population $Z$ and assume that for every $x$,
$d(x,Z)$ is bounded below by some positive power of the population
size.  For all $K>0$ and for all real bounded mesurable
function $f$ on $E$, if $Z=\sum_{i=1}^{N}\delta_{x^{(i)}}$ with $x^{(i)}=(x^{(i)}_1,x^{(i)}_2)$, the
generator of the Markov process $Z^K$ is:
\ba
L^Kf(Z)&=\sum_{x\in\mathbf{G}}b(x,Z)(f(Z+\delta_x)-f(Z))\\&+\sum_{i=1}^Nd(x_i,Z)(f(Z-\delta_{x^{(i)}})-f(Z))\\&
+\sum_{i=1}^N\frac{\mu}{K}\sum_{y\in\{\mathcal{A},\mathcal{C},\mathcal{G},\mathcal{T}\}^{G}}M(x^{(i)}_1,y)(f(Z-\delta_{x^{(i)}}+\delta_{(y,x^{(i)}_2)})-f(Z))\\&+\sum_{i=1}^N\frac{\mu}{K}\sum_{y\in\{\mathcal{A},\mathcal{C},\mathcal{G},\mathcal{T}\}^{G}}M(x^{(i)}_2,y)(f(Z-\delta_{x^{(i)}}+\delta_{(x^{(i)}_1,y)})-f(Z)).\ea
\textbf{Notations:} When the population is monomorphic, i.e. every individual has same genotype $x$, we assume that the population follows a neutral logistic birth-and-death process as presented in Section
\ref{sectioncasneutre}, and we denote by $b(x)$, $d(x)$ and $c(x)$ the birth, and natural and competition death rates (denoted $b$, $d$, and $c$ in Section \ref{sectioncasneutre}). For all demographic parameters $b$, $d$, and $c$, we also define the stationary law $l(.,b,d,c)$ of the population size of this neutral
logistic birth-and-death process. $l$ satisfies the stationary
equations system:
\be\left\{\begin{array}{l}b(N-1)l(N-1,b,d,c)+
(d+cN)(N+1)l(N+1,b,d,c)\\ ~\\\phantom{b(N-1)l(N-1,b,d,c)} =N(b+d+c(N-1))l(N,b,d,c)\quad\forall N\geqslant3\\ \\
2bl(2,b,d,c)=3(d+2c)l(3,b,d,c).
\end{array}\right.\ee Then for all $N\geqslant2$,
\ben\label{formulel}
l(N,b,d,c):=\frac{\displaystyle{\frac{1}{N}\prod_{k=2}^{N-1}\frac{b}{d+kc}}}
{\displaystyle{\sum_{i=2}^{\infty}\frac{1}{i}\prod_{j=2}^{i-1}\frac{b}{d+jc}}}.\een
We now rescale time when $K$ goes to infinity, in order to observe mutation apparitions. More precisely, the mean time of apparition of a mutation being equal to $1/\mu_K\sim K$, we accelerate time by multiplying $t$ by $K$.

\subsection{Convergence and limiting process in the adaptive dynamics asymptotics}
\begin{thm}\label{Zconverge}
For all $0<t_1<...<t_n$, the $n$-tuple $(Z^K_{Kt_1},...,Z^K_{Kt_n})$ converges in law towards the
process $(N_{t_1}\delta_{S_{t_1}},...,N_{t_n}\delta_{S_{t_n}})$ where
\begin{description}
\item[(i)] $(S_t)_{t>0}$ is a Markov jump process that jumps from a homozygous genotype $x^{(1)}=(x_1,x_1)$
to another homozygous genotype $x^{(2)}=(x_2,x_2)$ where
$x_1$ and $x_2$ are in $\{\mathcal{A},\mathcal{C},\mathcal{G},\mathcal{T}\}^{G}$, at rate
$\tau(x^{(1)},x^{(2)})$.
\item[(ii)] \ban\label{formuletaux}\tau(x^{(1)},x^{(2)})&=2\mu
M(x_1,x_2)\\
&\times\sum_{N=2}^{\infty}Nf((N-1,1,0),x^{(1)},x^{(2)})l(N,b(x^{(1)}),d(x^{(1)}),c(x^{(1)})),
\ean where $f((k,m,n),x^{(1)},x^{(2)})$ is the
probability that, starting from $k$ individuals with genotype
$x^{(1)}$, $m$ with genotype $(x_1,x_2)$, and $n$ with genotype $x^{(2)}$, the
population gets finally monomorphic with genotype $x^{(2)}$. In the
particular case where only the natural death rate differs between individuals with genotypes $x^{(1)}$ and $x^{(2)}$, as in Equation
$\eqref{equationtaux}$, \be
f((N-1,1,0),x^{(1)},x^{(2)})=u((N-1,1,0),d(x_1,x_2)-d(x^{(1)}),d(x^{(2)})-d(x^{(1)}))\ee
where $d(x^{(1)})$, $d((x_1,x_2))$, and $d(x^{(2)})$ are the
respective natural death rates of individuals with genotype $x^{(1)}$,
$(x_1,x_2)$ and $x^{(2)}$ (the generalization of genotypes $AA$,
$Aa$, and $aa$ in Section \ref{sectioncasneutre}), and $u$ has been
studied in Section \ref{sectionfixation}.
\item[(iii)] Conditionnally to $(S_{t_1},...,S_{t_n})=(x^{(1)},...,x^{(n)})$, the random variables $N_{t_1}$, ... ,
$N_{t_n}$ are mutually independent and for all $i$, $N_{t_i}$
has law $l(.,b(x^{(i)}),d(x^{(i)}),c(x^{(i)}))$.
\end{description}
\end{thm}
At this mutational time scale, the process $(N_t\delta_{S_t})_{t\geqslant0}$ describes the successive fixations of mutations. Indeed, a jump of the limiting process $S$ corresponds to a change in the genotype of every individual of the population, i.e. a mutation fixation. This previous theorem is directly obtained from \citet{ChampagnatLambert2007}, except from a few details in the proof, which are given in Appendix \ref{Appendicepreuve}.

\subsection{The extinction vortex}\label{sectionvortex}
In this section we focus on the jump process $S$ and assume that all mutations have the same effect
than described in Equation $\eqref{equationtaux}$, i.e. when $x_1$ mutates to $x_2$, individuals with genotypes $x^{(1)}$, $(x_1,x_2)$ and $x^{(2)}$ all have same fecundity $b$ and competition parameter $c$, but
\be d(x_1,x_2)=d(x^{(1)})+\delta, \quad \text{ and }\quad d(x^{(2)})=d(x^{(1)})+\delta'.\ee
What is more, we exclude overdominance cases by assuming that $\delta<\delta'$. We denote by
\ben\label{formuletauxd}\tau(d,\delta,\delta')=\sum_{N=2}^{\infty}Nu((N-1,1,0),d,\delta,\delta')l(N,d)\een
the jump rate of the limiting process
$S$ of Theorem \ref{Zconverge} (Equation $\eqref{formuletaux}$) when individuals have birth rate $b$,
natural death rate $d$, and competition rate $c$ (the dependence in parameters $b$ and $c$ is hidden, to simplify notations, we assumed $\mu=1/2$). This rate is also the rate of fixation of a deleterious mutation with size $(\delta,\delta')$. Let us recall that
the extinction vortex is due to more and more rapid fixations of
deleterious mutations in the population. We then wish to prove that the mean time to fixation of a deleterious mutation
decreases when the number of already fixed mutations increases. Now
when a deleterious mutation gets fixed, the natural death rate of
all individuals is increased by $\delta'$. The vortex is then due to the fact that the mean time to fixation of a deleterious mutation is a decreasing function of the natural death rate $d$ of individuals, which is proved in the next theorem.

\begin{thm} \label{tauxcroissant}
If $\delta>0$ and $\delta'>\delta$, and if $b$ is small enough, the mean
time to a jump of process $S$ $T(b,d,c,\delta,\delta')=1/\tau(b,d,c,\delta,\delta')$ is a decreasing function of $d$, the natural death rate of individuals.
\end{thm}
Here we underline the dependence of all quantities in $d$, by
denoting respectively by $u((k,m,n),d,\delta,\delta')$, $v((k,m,n),d)$, and
$v'((k,m,n),d)$ the fixation probability defined in Section \ref{sectionfixation} and its derivatives, when
individuals have natural death rate $d$. We also denote by $l(.,d)$
the stationary law of the population size (Equation
$\eqref{formulel}$). We first need to prove
the following lemma:
\begin{lem}\label{lemmel}
If $d$ and $d'$ are two non negative real numbers such that $d'>d$,
then there exists an integer $N_0$ such that for all $N\leqslant
N_0$, $l(N,d')\geqslant l(N,d)$, and for all $N>N_0$,
$l(N,d')<l(N,d)$.
\end{lem}
\begin{proof} Let us define $q(N)=\frac{l(N,d')}{l(N,d)}$.
Equation $\eqref{formulel}$ gives us that
$q(N+1)=\frac{d+cN}{d'+cN}q(N)$, then if $d'>d$, $q(N)$ is a
strictly decreasing function of $N$. Next,
\be
q(2)=\frac{\frac{1}{2}\sum_{i=2}^{\infty}\frac{1}{i}\prod_{j=2}^{\infty}\frac{b}{d+jc}}{\frac{1}{2}\sum_{i=2}^{\infty}
\frac{1}{i}\prod_{j=2}^{\infty}\frac{b}{d'+jc}},\ee hence $q(2)>1$.
Finally, if $q(N)>1$ for all $N$ then $l(N,d')>l(N,d)$ for
all $N$ which is absurd as $l(.,d)$ and $l(.,d')$ are
probability measures. Then there exists an integer $N_0$ such that for all
$N>N_0$, $q(N)<1$ and for all $N\leqslant N_0$, $q(N)\geqslant 1$.\end{proof}

\medskip
\begin{proof}{(Theorem \ref{tauxcroissant})} From Theorem \ref{maintheorem}, the mean time to fixation of a mutation is $T(d,\delta,\delta')=1/\tau(d,\delta,\delta')$ with
\ban\label{formuletauxvv'}
\tau(d,\delta,\delta')
&=\frac{1}{2}-\left[\sum_{N=2}^{\infty}N(\delta
v((N-1,1,0),d)+\delta'v'((N-1,1,0),d))l(N,d)\right]\\&+o(|\delta|+|\delta'|)\ean
where the differentiability of the infinite sum in $\eqref{formuletauxd}$ is obtained as in the proof of Proposition \ref{vSL}. Then if $d'>d$,
\ba \tau(d',\delta,\delta')-\tau(d,\delta,\delta')&=
\sum_{N=2}^{\infty}N(\delta v((N-1,1,0),d)+\delta'
v'((N-1,1,0),d))l(N,d)\\&-\sum_{N=2}^{\infty}N(\delta
v((N-1,1,0),d')+\delta'v'((N-1,1,0),d'))l(N,d')\\&+o(|\delta|+|\delta'|)\\&
=\delta\sum_{N=2}^{\infty}Nl(N,d)(v((N-1,1,0),d)-v((N-1,1,0),d'))\\&
-\delta\sum_{N=2}^{\infty}Nv((N-1,1,0),d')(l(N,d')-l(N,d))\\&
+\delta'\sum_{N=2}^{\infty}Nl(N,d)(v'((N-1,1,0),d)-v'((N-1,1,0),d'))\\&
-\delta'\sum_{N=2}^{\infty}Nv'((N-1,1,0),d')(l(N,d')-l(N,d))\\&
+o(|\delta|+|\delta'|).\ea
Defining $N_0$ as in Lemma \ref{lemmel}, we obtain:
\ban\tau(d',\delta,\delta')&-\tau(d,\delta,\delta')=\delta
\sum_{N=2}^{\infty}Nl(N,d)(v((N-1,1,0),d)-v((N-1,1,0),d'))\\&
+\delta'\sum_{N=2}^{\infty}Nl(N,d)(v'((N-1,1,0),d)-v'((N-1,1,0),d'))\\&
-\delta\sum_{N=2}^{\infty}(Nv((N-1,1,0),d')-N_0v((N_0-1,1,0),d'))(l(N,d')-l(N,d))\\&
-\delta'\sum_{N=2}^{\infty}(Nv'((N-1,1,0),d')-N_0v'((N_0-1,1,0),d'))(l(N,d')-l(N,d))\\&
+o(|\delta|+|\delta'|) \quad\text{the added terms being equal to $0$},\ean
which gives, if $w((k,m,n),d)=\delta v((k,m,n),d)+\delta'v'((k,m,n),d)$,
\ban\label{differencetaux}\tau(d',\delta,\delta')&-\tau(d,\delta,\delta')=\sum_{N=2}^{\infty}Nl(N,d)(w((N-1,1,0),d)-w((N-1,1,0),d'))\\&
-\sum_{N=2}^{\infty}(Nw((N-1,1,0),d')-N_0w((N_0-1,1,0),d'))(l(N,d')-l(N,d))\\&+o(|\delta|+|\delta'|)\ean
Let us now prove first that $N\mapsto Nw((N-1,1,0),d')$ is increasing and then that $d\mapsto (w((N-1,1,0),d)$ is decreasing. These two results imply Theorem \ref{tauxcroissant} and will be consequences of the two following lemmas.
Notice that the infinitesimal generator $L$ (Equation $\eqref{formuleL}$) is
the sum of two generators $$(Lf)(k,m,n)=(L_bf)(k,m,n)+(L_df)(k,m,n)$$ where
\ba L_bf(Z)&=\sum_{i=1}^3b_i(Z)(f(Z+e_i)-f(Z)), \quad\text{and}\\
L_df(Z)&=(d+c(N-1))\\&\times[kf(k-1,m,n)+mf(k,m-1,n)+nf(k,m,n-1)-Nf(k,m,n)].\ea
Since $\partial Lw/\partial d=0$ (from $\eqref{Deltavmin}$ and $\eqref{deltav'}$),
\ben \label{Ldv}\left(L\frac{\partial w(.,d)}{\partial
d}\right)(k,m,n)=\frac{-(L_dw(.,d))(k,m,n)}{d+c(N-1)}.\een
Notice also that \ben \label{equationLdwN-1}(L_dw(.,d))(N-1,1,0)=(d+c(N-1))[(N-1)w(N-2,1,0,d)-Nw(N-1,1,0,d)],\een so if we prove that $(L_dw(.,d'))(N-1,1,0)\leqslant0$ for all $N\geqslant2$, then $N\mapsto Nw((N-1,1,0),d')$ is increasing. In fact we prove the
\begin{lem}\label{signeLdv} If $b$ is small enough and $\delta'>\delta$, then for all
$(k,m,n)$ in $\mathbb{N}^{3}_{**}$, $$\left(L\frac{\partial
w(.,d)}{\partial d}\right)(k,m,n)\geqslant0.$$
\end{lem}
\begin{proof}{(Lemma \ref{signeLdv})} There exists a constant $C>0$ such that for all $(k,m,n)$ in $\mathbb{N}^{3}_{**}$,
\ba
(Lw(.,d))(k,m,n)
&=(L_dw(.,d))(k,m,n)\left(1+\frac{b}{d+c(N-1)}\right)\\&+\left((L_bw(.,d))(k,m,n)-\frac{b}{d+c(N-1)}(L_dw(.,d))(k,m,n)\right)
\\&=-\frac{\delta m(k-n)+\delta'nY}{2N(N-1)}=-\frac{(\delta'-\delta)nm+k(\delta m+2\delta'n)}{2N(N-1)}\\&\leqslant\frac{-C(km+mn+kn)}{2N(N-1)}.\ea
Next, detailed computations give us that there exists a constant $C'$ such that \ba &\left|(L_bw(.,d))(k,m,n)-\frac{b}{d+c(N-1)}(L_dw(.,d))(k,m,n)\right|\\&\leqslant \delta b\left[|k-n|\left(m|x_{N+1}-x_{N-1}|+
\frac{N^2-(k-n)^2}{N}|y_{N+1}-y_{N-1}|\right)\right]\\&
+\delta'b[Ym|x_{N+1}-x_{N-1}|+mN|x'_{N+1}-x'_{N-1}|
+(2N-Y)Y|y_{N+1}-y_{N-1}|\\&\phantom{\delta'b[Ymx}+(2N-Y)Y|y'_{N+1}-y'_{N-1}|)]\\&
+bC'\frac{km+mn+kn}{N}(|x_{N+1}|+|x_{N-1}|+|x'_{N+1}|+|x'_{N-1}|\\&\phantom{+bC'(km+mn+kn)(xN+1}+|y_{N+1}|+|y_{N-1}|+|y'_{N+1}|+|y'_{N-1}|)\ea
Finally, from Equations \eqref{zNfinal} and \eqref{zN1surN}, when $b$ is small enough, there exists a constant $C''$ independent from $b$ such that $|x_{N+1}|<\frac{C''}{N}$, and the same result is true for $y$, $x'$ and $y'$. Then if $b$ is small enough, $$\left|L_bw(k,m,n)-\frac{b}{d+c(N-1)}L_dw(k,m,n)\right|<\frac{C(km+mn+kn)}{2N(N-1)} \quad \forall(k,m,n)\in\mathbb{N}^3_{**}$$
which gives that $Ldw(k,m,n)\leqslant0$ for all $(k,m,n)$ and the result by $\eqref{Ldv}$.\end{proof}

\medskip
We finally prove that
\begin{lem}\label{LemmaLdv} If $b$ is small enough and $\delta'>\delta$, then for all $(k,m,n)$ in $\mathbb{N}^3_{**}$,
\ben\frac{\partial w((k,m,n),d)}{\partial
d}=-\mathbb{E}_{(k,m,n)} \int_0^{T_{\Gamma}}\left(L\frac{\partial
w(.,d)}{\partial d}\right)(Z_t)dt.\een
\end{lem}
\begin{proof}{(Lemma \ref{LemmaLdv})} We use Dynkin's formula,
stopped at time $T_N=\inf\{t>0,N_t\geqslant N\}$: \be \frac{\partial
w(Z_{T_{\Gamma}\land T_N},d)}{\partial d}= \frac{\partial w(Z_0,d)}{\partial
d}+M_{T_{\Gamma}\land T_N}+\left[\int_0^{T_{\Gamma}\land T_N}\left(L\frac{\partial
w(.,d)}{\partial d}\right)(Z_s)ds\right],\ee where $(M_{t\land
T_N})_{t>0}$ is a martingale. Since $(L\partial w/\partial
d(.,d))(k,m,n)\geqslant0$ for all $(k,m,n)$ (Lemma \ref{signeLdv}), then if $k+m+n=N_0$, $$\left(\int_0^{T_{\Gamma}\land
T_N}\left(L\frac{\partial w(.,d)}{\partial
d}\right)(Z_s)ds\right)_{N\geqslant N_0}$$ and $$\left(\partial w(Z_{T_{\Gamma}\land
T_N},d)/\partial d-\partial w(Z_0,d)/\partial d-M_{T_{\Gamma}\land
T_N}\right)_{N\geqslant N_0}$$ are two increasing sequences of positive
variables since $T_N\leqslant T_{N+1}$ when $N\geqslant N_0=k+m+n$.
From the monotone convergence theorem, since $T_{\Gamma}\land
T_N\underset{N\rightarrow\infty}{\longrightarrow} T_{\Gamma}$ p.s.
(Proposition \ref{sup}), \be
\mathbb{E}_{(k,m,n)}\left[\int_0^{T_{\Gamma}\land T_N}\left(L\frac{\partial
w(.,d)}{\partial
d}\right)(Z_s)ds\right]\underset{N\rightarrow\infty}{\longrightarrow}
\mathbb{E}_{(k,m,n)}\left[\int_0^{T_{\Gamma}}\left(L\frac{\partial
w(.,d)}{\partial d}\right)(Z_s)ds\right]\ee
and \ba
\mathbb{E}_{(k,m,n)}&\left(\frac{\partial w(Z_{T_{\Gamma}\land
T_N},d)}{\partial d}-\frac{\partial w(Z_0,d)}{\partial d}-M_{T_{\Gamma}\land
T_N}\right)\\&\phantom{\frac{\partial w(Z_{T_{\Gamma}\land
T_N},d)}{\partial d}-\frac{\partial w(Z_0,d)}{\partial d}}\underset{N\rightarrow\infty}{\longrightarrow}\mathbb{E}_{(k,m,n)}
\left[\frac{\partial w(Z_{T_{\Gamma}},d)}{\partial d}-\frac{\partial
w(Z_0)}{\partial d}-M_{T_{\Gamma}}\right].\ea
Using $\partial w(Z_{T_{\Gamma}},d)/\partial d=M_{T_{\Gamma}}=0$, we get the
result.\end{proof}

\medskip
Finally, $\eqref{Ldv}$, $\eqref{equationLdwN-1}$ and Lemma \ref{signeLdv} imply that $N\mapsto Nw((N-1,1,0),d)$ is an increasing function of $N$, and Lemmas \ref{signeLdv} and \ref{LemmaLdv} give
that $w((N-1,1,0),d)$ is a decreasing function of $d$.\end{proof}

\subsection{Numerical results}\label{sectionnumeriquemutationnel}
Equation $\eqref{zNfinal}$ allows us to
approximate the sequences $(z_N)_{N\geqslant2}$ numerically, and we do the same for
$(z'_N)_{N\geqslant2}$ and then for $\tau$ (Equation $\eqref{formuletauxvv'}$). Figure \ref{figureT} shows the
mean time $T$ to fixation of a deleterious mutation as a decreasing
function of $d$ (Theorem \ref{tauxcroissant}), for various values of
$b$, $\delta$, and $\delta'$.
\begin{figure}[ht]\center
\subfloat[][]{\label{figureTd}
\scalebox{0.45}{\includegraphics[trim=0cm 0cm 0cm
0cm,clip]{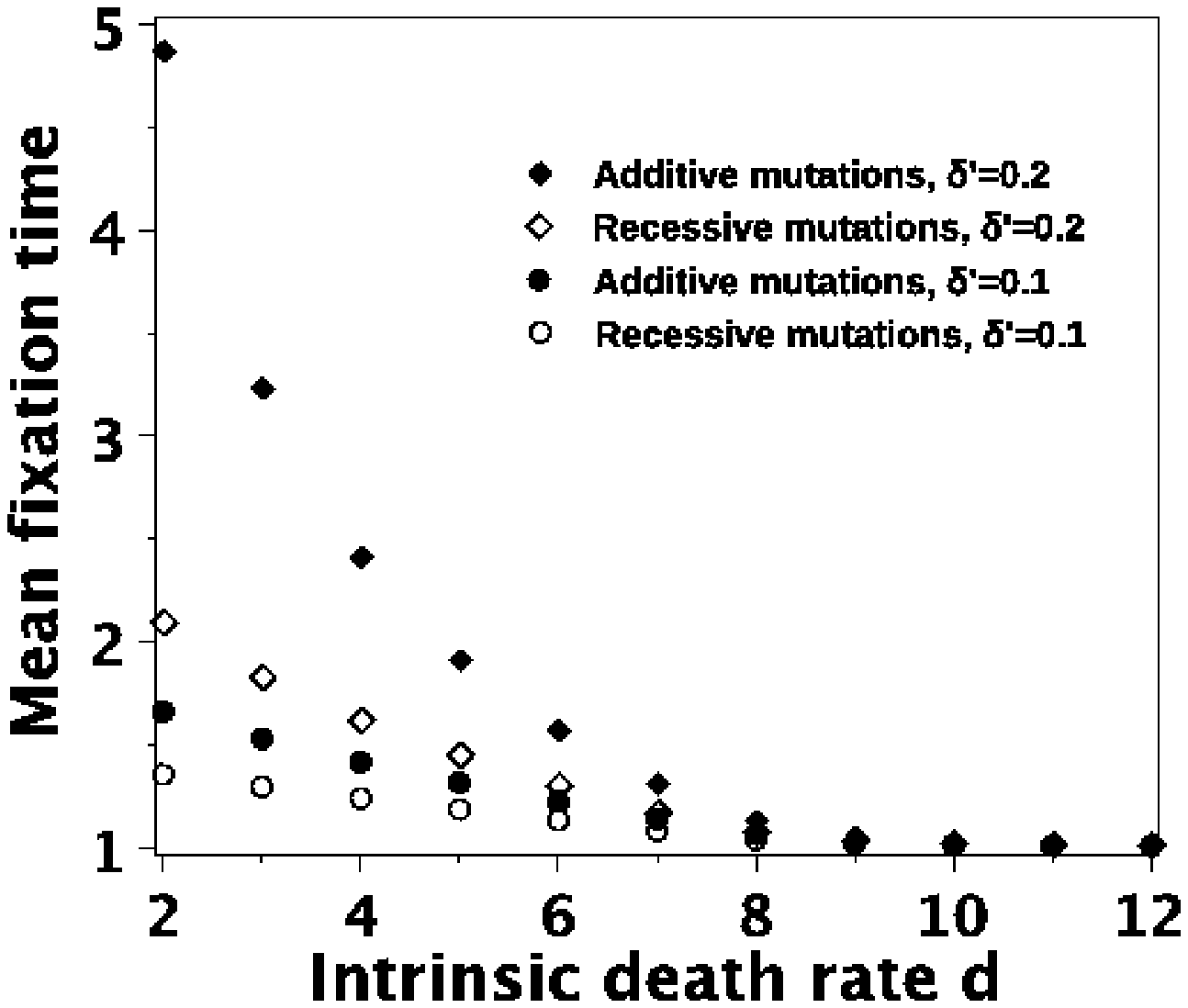}}}\quad \subfloat[][]{\label{figuredtbd}
\scalebox{0.455}{\includegraphics[trim=0cm -0.2cm 0cm
0cm,clip]{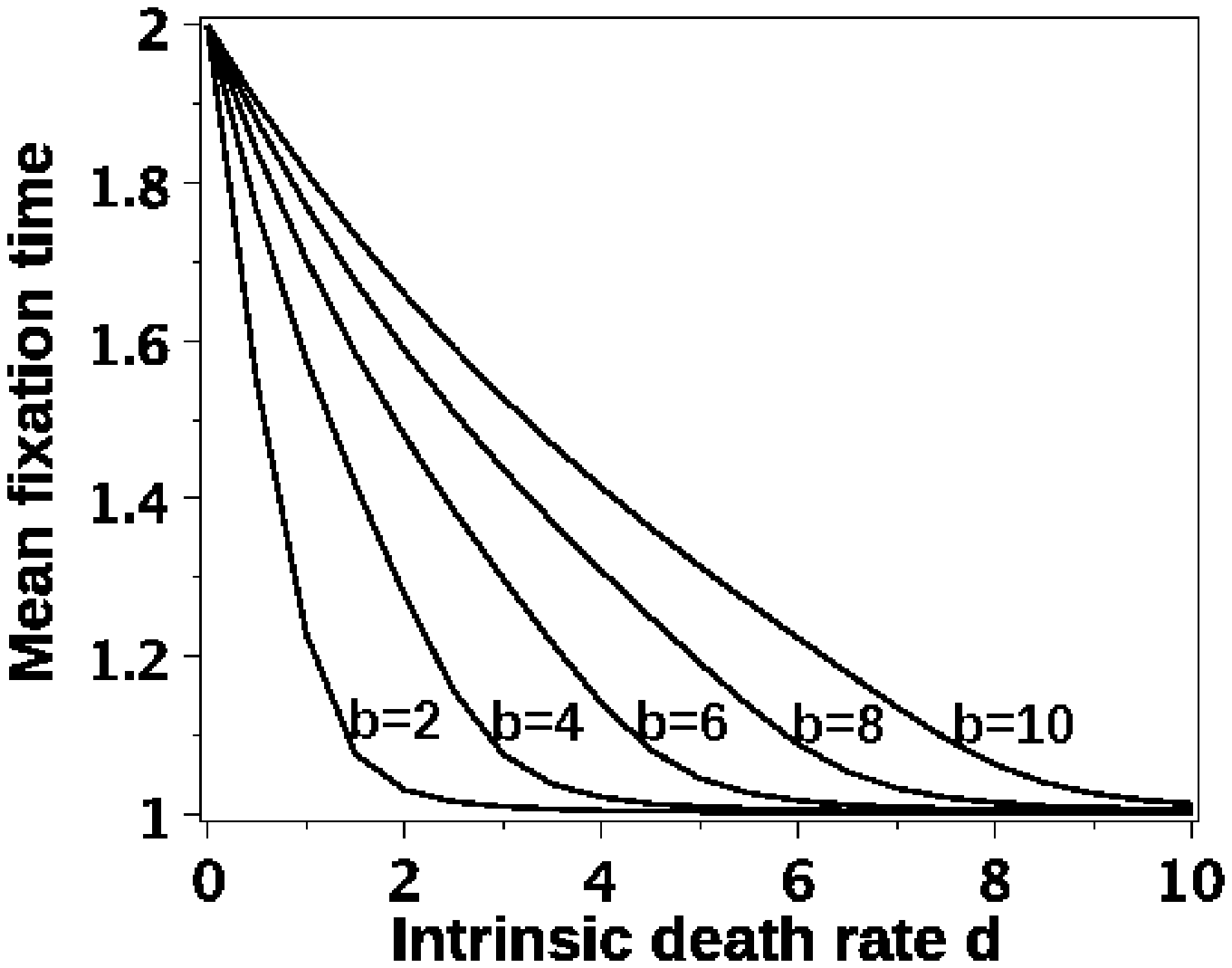}}} \caption[]{\subref{figureTd}:
Relationship between $T$, the mean time to fixation of a deleterious
mutation, and the population intrinsic death rate $d$ as a function
of selection and dominance. Open symbols: recessive mutation
($\delta=0$); closed symbols: additive mutation
($\delta=\delta'/2$); circles: $\delta'=0.1$; diamonds:
$\delta'=0.2$. Other demographic parameters are $b=10$, $c=0.1$, and
$m=1$.\subref{figuredtbd}: Relationship between the mean time to
fixation of a deleterious mutation $T$ and parameters $b$ and $d$.
Each curve corresponds to a fixed value of $b$. Other parameters are
$\delta=0.05$, $\delta'=0.1$, $c=0.1$ and $m=1$.}\label{figureT}
\end{figure}
For more biological analysis and numerical results, we refer to \citet{Coronbio2012}.

\appendix

\section{Proof of Theorem \ref{Zconverge}}\label{Appendicepreuve}

In this article we consider a diploid population and, as seen in Theorem \ref{maintheorem}, the diploidy generates interesting formulas for the fixation probability of a non neutral allele. More precisely, this fixation probability is a function of the initial genetic repartition in the population (parameters $k$, $m$, and $n$) and cannot be reduced to a function of the initial numbers of allele $A$ and $a$ in the population, as for a haploid population. At the mutational time scale (Section \ref{sectionmutationnelle}), this leads to
mutation fixation rates that are different than those obtained in \citet{ChampagnatLambert2007} for the haploid case.

However, the proof of Theorem \ref{Zconverge} can be seen as an extension of the proof of Theorem $3.1$ of
\citet{ChampagnatLambert2007}, to the cases where mutations occur during
life and not at birth, and where no death can occur when there
are two individuals in the population. We now explain why those
differences do not hamper the proof of Theorem $3.1$ of
\citet{ChampagnatLambert2007}, which is constituted of three
lemmas.

\medskip
\textbf{First lemma:} Lemma $6.2$ of \citet{ChampagnatLambert2007} proves that there are no mutation accumulations
when parameter $K$ goes to infinity. Using Proposition \ref{sup}, the lemma and its proof remain true in our model.

\medskip
\textbf{Second lemma:} The first part of Lemma $6.3$ of \citet{ChampagnatLambert2007} gives the limiting
law of $K\tau_1$ and of the population size at time $\tau_1$ when
$K$ goes to infinity, where $\tau_1$ is the first mutation
apparition time for the population $Z^K$. Here the proof is similar but uses different rates: as long as
$t<\tau_1$, if the population is initially monomorphic with genotype
$x$, the population size $(N_t^K)_{0<t<\tau_1}$ follows a birth and
death process with birth rate $b(x,i\delta_x)i$ and death rate
$d(x,i\delta_x)i$ when $N_t^K=i$, and  $\tau_1$ is the first point of an
inhomogeneous Poisson point process with intensity $(2\mu/K) N^K_t$.
Then for any bounded function $f:\mathbb{N}\setminus\{0,1\}\rightarrow\mathbb{R}$,
\ba\mathbb{E}(f(N^K_{\tau_1^-})\mathbf{1}_{\{t\geqslant\tau_1/K\}})
&=2\mu\int_0^t\mathbb{E}(f(N^K_{Ks})N^K_{Ks}
e^{-2\mu/K\int_0^{Ks}N_u^Kdu}ds)\\
&=2\mu\int_0^t\mathbb{E}(f(N^0_{Ks})N^0_{Ks}
e^{-2\mu/K\int_0^{Ks}N_u^0du}ds)\ea since the law of $N^K_t$ does not depend on $K$. The ergodic theorem finally
gives us that
\be\lim_{K\rightarrow\infty}\mathbb{E}^K(f(N^K_{\tau_1^-})
\mathbf{1}_{\{t\geqslant\tau_1/K\}})=
\frac{\mathbb{E}(Nf(N))}{\mathbb{E}(N)}\int_0^t2\mu\mathbb{E}(N)e^{-2\mu\mathbb{E}(N)s}ds\ee
where $N$ is a random variable with law $l$ defined by $\eqref{formulel}$.
The second part of Lemma $6.3$ of \citet{ChampagnatLambert2007} gives us that
$sup_{K>1}\mathbb{E}^K_{n\delta_x}(N_{\tau_1}^p)<\infty$. Here the proof needs to be slightly changed as the population
size does not reach $1$ in our model. We then define
$L_t=\int_0^t \mathbf{1}_{\{N^0_u=2\}}du$ and have
\be\mathbb{E}^K_{n\delta_x}(N_{\tau_1}^p)\leqslant2\mu\int_0^{\infty}\mathbb{E}(N^{p+1}_{Ks}exp(-\frac{2\mu}{K}L_{Ks})
ds).\ee We finally prove that there exist $\lambda$, $\lambda'$,
$C>0$ such that $\mathbb{P}(L_t\leqslant \lambda t) \leqslant
Ce^{-\lambda' t}$ as in \citet{ChampagnatLambert2007}, by defining
$s_i:=\inf\{s\geqslant t_{i-1}:N_s^0=2\}$ and $t_i=\inf\{t\geqslant
s_i:N^0_s=3\}$.

\medskip
\textbf{Third lemma:} The third lemma gives the behavior of $\rho_1$, the
first time where the population becomes monomorphic, and $V_1$, the
genotype of individuals at time $\rho_1$, if the population
initially contains $2$ genotypes $x$ and $y$. This lemma and the end
of the proof of Theorem \ref{Zconverge} are easily generalized to
our model. $\blacksquare$

\bigskip

\textbf{Acknowledgements:} I fully thank my Phd director Sylvie M\'el\'eard for her constructive comments and continual guidance during my work. This article benefited from the support of
the ANR MANEGE (ANR-09-BLAN-0215) and from the Chair ``Mod\'elisation
Math\'ematique et Biodiversit\'e" of Veolia Environnement - \'Ecole
Polytechnique - Museum National d'Histoire Naturelle - Fondation X.

\bibliographystyle{plainnat}
\bibliography{mabiblio}

\end{document}